\documentclass[reqno,12pt]{amsart}

\usepackage{graphicx, amsmath, amssymb, amscd, amsthm, euscript, psfrag, amsfonts,bm}

\theoremstyle{plain}

\newtheorem{thm}[equation]{Theorem}

\newtheorem{rmk}[equation]{Remark}
\newtheorem{cor}[equation]{Corollary}
\newtheorem{prop}[equation]{Proposition}

\newtheorem{lem}[equation]{Lemma}
\newtheorem{dfn}[equation]{Definition}
\numberwithin{equation}{section}

\newcommand{\sm}{\left(\begin{smallmatrix}}
\newcommand{\esm}{\end{smallmatrix}\right)}
\newcommand{\bpm}{\begin{pmatrix}}
\newcommand{\ebpm}{\end{pmatrix}}

\newcommand{\C}{\mathbb{C}}
\newcommand{\G}{\mathbb{G}}

\newcommand{\Q}{\mathbb{Q}}

\newcommand{\R}{\mathbb{R}}
\newcommand{\Z}{\mathbb{Z}}
\newcommand{\A}{\mathcal{A}}

\newcommand{\bH}{\mathbb{H}}

\newcommand{\bsl}{\backslash}

\newcommand{\op}{\operatorname}
\newcommand{\BR}{\op{BR}}\newcommand{\BMS}{\op{BMS}}\newcommand{\SO}{\op{SO}}
\newcommand{\PSU}{\op{PSU}}
\renewcommand{\G}{\Gamma}\newcommand{\PS}{\op{PS}}\newcommand{\PSL}{\op{PSL}}\newcommand{\T}{\op{T}}
\newcommand{\ba}{\backslash}
\newcommand{\e}{\epsilon}
\newcommand{\z}{\Z}\newcommand{\br}{\R}
\renewcommand{\c}{\C}
\newcommand{\la}{\langle}\newcommand{\Res}{\op{Res}}

\renewcommand{\P}{\mathcal P}
\newcommand{\ra}{\rangle}

\newcommand{\be}{\begin{equation}}
\newcommand{\ee}{\end{equation}}

\newcommand{\s}{\bf s}\newcommand{\mC}{\mathcal C}

\begin{document}

\title[Effective count for Apollonian circle packings]
{Effective circle count for Apollonian packings and
closed horospheres}

\author{Min Lee and Hee Oh}

\address{Mathematics department, Brown university, Providence, RI}\email{minlee@math.brown.edu}
\address{Mathematics department, Brown university, Providence, RI
and Korea Institute for Advanced Study, Seoul, Korea}\email{heeoh@math.brown.edu}

\thanks{The authors are respectively  supported in parts by Simons Fellowship and by NSF Grant \#1068094.}

\begin{abstract}
The main result of this paper is an effective count for Apollonian circle packings that are either
bounded or contain two parallel lines.
We obtain this by proving an effective equidistribution of closed horospheres
in the unit tangent bundle of a geometrically finite hyperbolic $3$-manifold, whose fundamental
group has critical exponent bigger than $1$. We also discuss applications to Affine sieves.
Analogous results for surfaces are treated as well.
%Non-effective versions of equidistribution and counting in this context were known previously. 
\end{abstract}

\maketitle
%\tableofcontents
\section{Introduction}
\subsection{Apollonian circle packings}
An Apollonian circle packing is an ancient Greek construction which
is made by repeatedly inscribing circles into the triangular interstices of four mutually tangent circles in the plane.
In recent years, there have been many new and exciting developments
in the study of Apollonian circle packings; for instance,
 see \cite{G1}, \cite{G2}, \cite{S}, \cite{BGS2}, \cite{KO}, \cite{BF}, \cite{BK},
 \cite{OS}, \cite{OhShahcircle}, \cite{OhICM}, etc.

The main goal of this paper is to obtain an effective version of 
 the counting theorem for circles in an Apollonian packing with bounded curvature. 

Let $\P$ be an Apollonian circle packing, that is either bounded or lies between two parallel lines
(i.e., congruent to the packing in Figure \ref{f2}).  
For $T>0$ and $\P$ bounded, we define the circle counting function as follows:
$$N_T(\P):=\#\{C\in \P: \text{Curv}(C)<T\}$$
 where $\text{Curv}(C)$ denotes the curvature of $C$, i.e., the reciprocal of the radius of $C$.
 For $\P$ unbounded between two parallel lines, we adjust the definition of $N_T(\P)$
to count circles only in a fixed period. %Note that $N_T(\P)<\infty$ for each $T>0$.

\medskip The main term in the asymptotic for $N_T(\P)$ will be described in terms of
the residual set of $\P$ (=the closure of the union of all circles in $\P$), denoted by $\Res(\P)$.
We denote by $\alpha$ the Hausdorff dimension of $\Res(\P)$; $\alpha$ is independent of $\P$ and known to be approximately
$1.30568(8)$ \cite{Mc}.
Let $\mathcal H^\alpha(\Res(\P))$
be the $\alpha$-dimensional Hausdorff measure of $\Res(\P)$ for bounded $\P$.
 For $\P$ between two parallel lines, we let $\mathcal H^\alpha(\Res(\P))$ be the measure of $\Res(\P)$ in a fixed period.

The error term in our asymptotic formula depends directly on the $L^2$-spectral gap of 
the complete hyperbolic $3$ manifold whose fundamental group
is the symmetry group of $\P$. The group $\PSL_2(\c)$ acts on the extended plane by linear fractional
 transformations. Set $$\mathcal A_\P:=\{g\in \PSL_2(\c): g(\P)=\P\} .$$
It is known that $\mathcal A_P$ is a geometrically finite discrete subgroup of $\PSL_2(\c)$ with 
critical exponent equal to $\alpha$ (cf. \cite{KO}). The fact $\alpha$ is strictly bigger than $1$ yields that
$\alpha(2-\alpha)$ is the smallest eigenvalue of the Laplacian $\Delta$ on the $L^2$-spectrum of the hyperbolic
manifold $\mathcal A_\P\ba \bH^3$ by Sullivan \cite{Su} and is also isolated by Lax and Phillips \cite{LP}.
Hence there exists $1<s_1<\alpha$ such that there is no eigenvalue of $\Delta$ in $L^2(\mathcal A_\P\ba \bH^3)$
 between $\alpha(2-\alpha)$ and $s_1(2-s_1)$.
Since all $\A_\P$'s are conjugate to each other by elements of $\PSL_2(\c)$, $s_1$ is independent of $\P$.

\medskip

Our effective counting result, which is a special case of our more general theorem (Theorem \ref{ec2m}),
 can be stated as follows:
\begin{thm} \label{m} As $T\to \infty$,
$$N_T(\P)= c_{A} \cdot \mathcal H^{\alpha}(\Res (\P)) \cdot T^{\alpha} + O(T^{\alpha -\tfrac{2(\alpha -s_1)}{63}})$$
where $c_A>0$ is a constant independent of $\P$.
\end{thm}

\begin{rmk}\rm
 \begin{enumerate}
  \item In \cite{KO}, the asymptotic
$ N_T(\P)\sim c_{\mathcal P} \cdot T^{\alpha} $ was obtained with less clear interpretation of the constant $c_{\mathcal P}$.

\item A similar type of asymptotic formula  
was obtained in \cite{OhShahcircle} for all Apollonian packings (whether bounded or not) by counting circles
in a bounded region, but with no error term.

\item
There are several different ways of understanding the constant
$c_{A} \cdot \mathcal H^{\alpha}(\Res (\P))$ in front of the main term, due to different approaches to the counting problem.
One description is given in our paper (see \eqref{cpp}). The aforementioned paper \cite{OhShahcircle} gives
another expression as well.

\item An Apollonian packing $\P$ is called integral if the curvatures of all circles in $\P$ are integers.
Any integral Apollonian packing is known to be either bounded or lies between two parallel lines. Therefore
Theorem \ref{m} applies to all integral Apollonian packings.
 \end{enumerate}

\end{rmk}

\begin{figure} \includegraphics [width=2in]{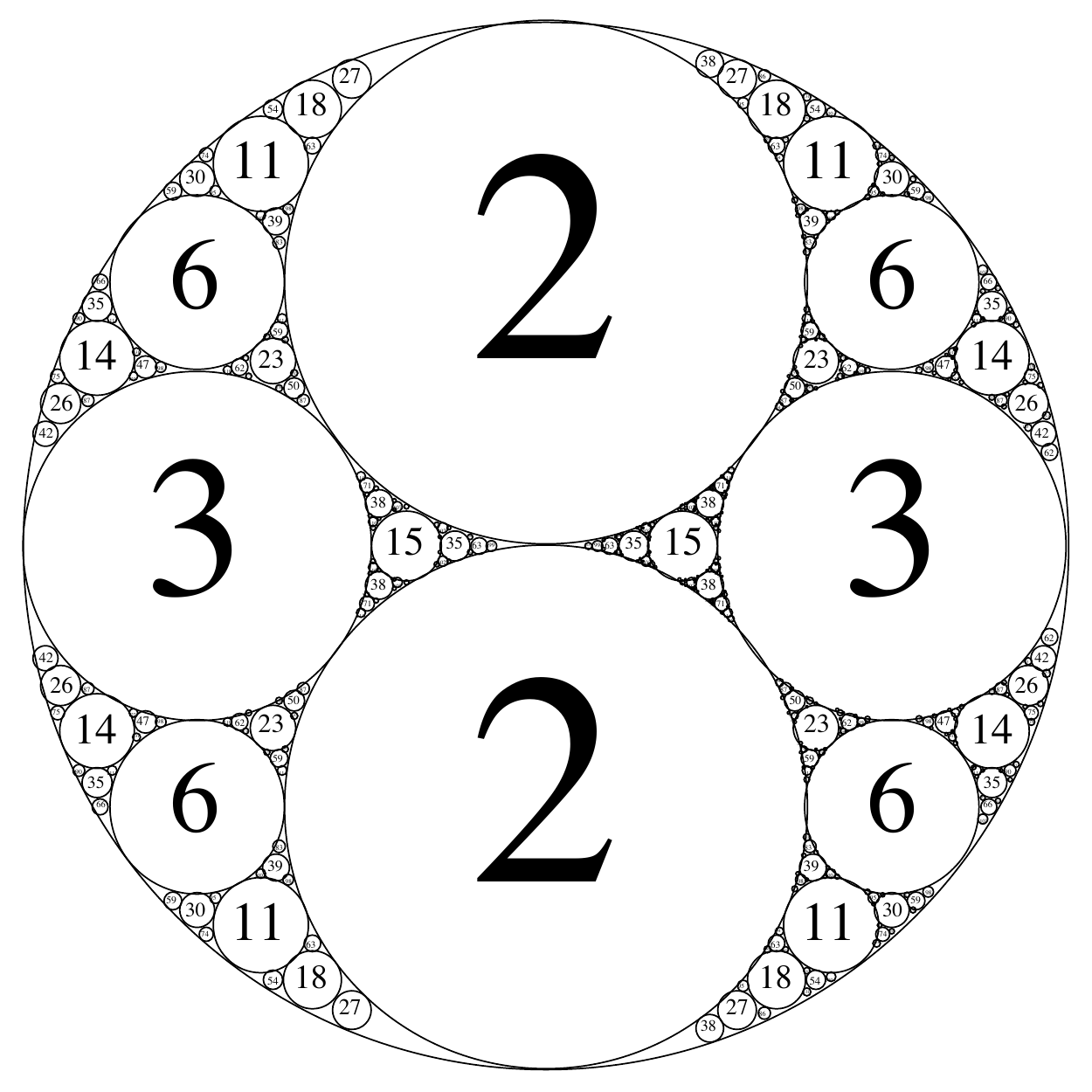}
\caption{A bounded Apollonian circle packing.}\label{f1}\end{figure}

\begin{figure} \includegraphics[width=2in] {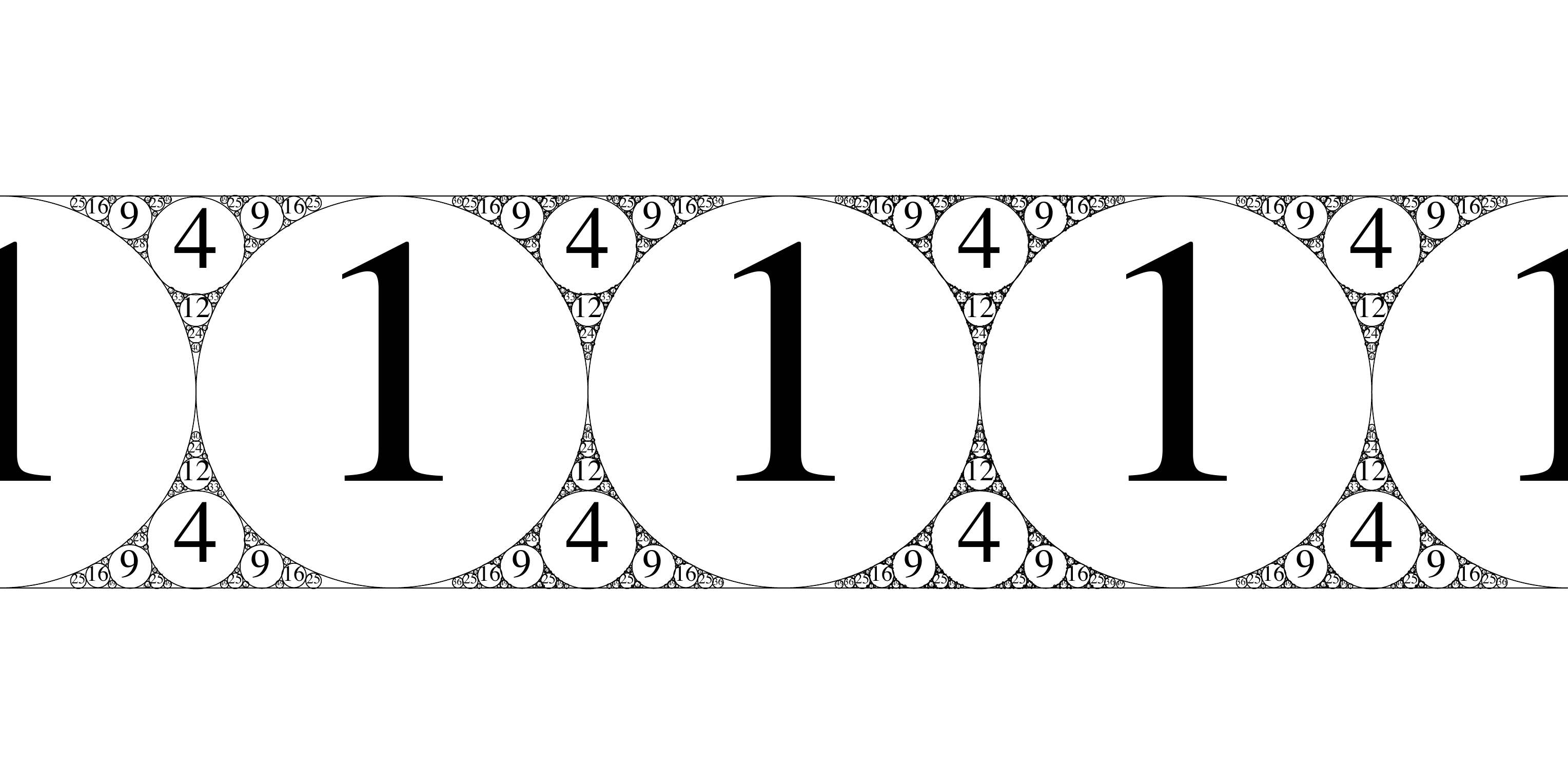}
\caption{An unbounded Apollonian circle packing bounded by two parallel line.}
\label{f2}\end{figure}

Based on the Descartes circle theorem \cite{Co}, the approach in \cite{KO} 
was to relate the circle counting problem with the equidistribution of closed
horospheres in the unit tangent bundle of the hyperbolic manifold $\A_\P\ba \mathbb H^3$.

The new achievement of this paper is an {\it effective} equidistribution of 
closed horospheres (Theorem \ref{main}). Besides its application to counting problems,
such equidistribution result is of independent interest 
in homogeneous dynamics.

\subsection{Effective equidistribution of  closed horospheres}
We obtain an effective equidistribution for closed horospheres in the unit tangent bundle
of hyperbolic $n$-manifolds for $n=2$ or $3$.
Consider the upper half space $\bH^n=\{(x,y): x \in \br^{n-1}, y>0\}$ with the metric given by
$ds^2= \tfrac{\sum_{i=1}^{n-1}dx_i^2 +dy^2}{y^2}$ and let $G=\op{Isom}^+(\bH^n)$
denote the group of orientation preserving isometries of $\bH^n$. 
That is, $G=\PSL_2(\br)$ for $n=2$ and $G=\PSL_2(\c)$ for $n=3$.
 
Let $\G<G$ be a torsion-free
discrete subgroup, which is not virtually abelian. We assume that $\G$ is geometrically finite,
that is, it admits a finite sided fundamental domain in $\bH^n$.
The limit set $\Lambda(\G)$ is the subset of the boundary $\partial(\bH^n)= \br^n\cup \{\infty\}$
consisting of all accumulation points in an orbit $\Gamma (z)$, $z\in \bH^n$.
 We denote by $0<\delta\le n-1$ the critical exponent of $\G$; it is
equal to the Hausdorff dimension of $\Lambda(\G)$
\cite{Sullivan1984}. 
%We will assume that $\delta<2$, or equivalently $\G\ba \bH^3$ is of infinite volume.

\medskip
For $G=\PSL_2(\br)$, set $K:=\op{PSO}(2)$,  and for $G=\PSL_2(\c)$, set $K:=\PSU(2)$.
In both cases, set
$$ A:=\left\{a_y:= \begin{pmatrix} \sqrt y &0 \\  0 & \tfrac{1}{\sqrt y}\end{pmatrix}: y>0\right\},$$ 
and let $M$ be the centralizer of $A$ in $K$.

The hyperbolic manifold $\G\ba \bH^n$ and
its unit tangent bundle $\T^1(\G\ba \bH^n)$  can be 
 identified with the double quotient spaces $\G\ba G/K$ and
 $\G \ba G/M$ respectively. Accordingly, functions on $\G\ba \bH^n$ (resp. $\T^1(\G\ba \bH^n)$)
 can be considered as right $K$-invariant (resp. $M$-invariant) functions on $\G\ba G$.
Since $a_y$ commutes with $M$,
$a_y$ acts on $\G\ba G/M$ by the multiplication from the right and
 this action corresponds to the geodesic flow in the unit tangent bundle. 

Set $N=\{g\in G: a_y^{-1} g a_y\to e\text { as $y\to \infty$}\}$; the contracting horospherical subgroup under the action of $a_y$.
Setting $$n_x:=\begin{pmatrix} 1& x\\ 0 & 1\end{pmatrix}$$
we have $N=\{n_x: x \in \br\}$ for $G=\PSL_2(\br)$, and $N=\{n_x:x\in \c\}$ for $G=\PSL_2(\c)$.
Even though there is no
action of $N$ on $\G\ba G/M$, the $N$-orbits $\{[g]N:=\G\ba \G gMN /M :g\in G\}$ are well-defined since $N$ is normalized
by $M$; these orbits give rise to the 
stable horospherical foliation of  $\T^1(\G\ba \bH^n)$.

%Fixing closed horosphere $[g]N$ in $X$, we prove an effective equidistribution of horospheres
% in $\G\ba G/M$ in the case when $1<\delta<2$. 
%To state our result precisely, we assume that $1 < \delta < 2 $ 
%and 
\medskip

In the rest of the introduction, we assume that $(n-1)/2<\delta<n-1$
 and that $\G\ba \G N$ is closed in $\G\ba G$. 
 In particular, $\G$ has infinite covolume in $G$. 
By the torsion-free assumption on $\G$,
$\G\cap NM=\G\cap N$ and we can identify $\G\ba \G NM/M $ with $(\G\cap N)\ba N$.
Note that the quotient $(\G\cap N)\ba N$ can be naturally identified
with $ (\br/\z)^{k} \times \br^{n-1-k}$ where $0\le k\le n-1$ denotes the rank
 of the free abelian subgroup $\G\cap N$.

\subsection{Equidistribution in spectral terms}
 We describe the effective equidistribution of $\G\ba \G N a_y$ as $y\to 0$
in $\T^1(\Gamma\ba \bH^n)$ in terms of 
the $M$-invariant spectrum of $L^2(\G\ba G)$ for a Casimir element of $G$.

By Lax and Phillips \cite{LP} and Sullivan \cite{Su}, the Laplacian $\Delta$ on $L^2(\G\ba \bH^{n})$
 has only finitely many eigenvalues $$0<\alpha_0=\delta(n-1-\delta)<\alpha_1\le \cdots
\le \alpha_k <\tfrac{(n-1)^2}{4}$$ lying below the continuous spectrum $[\tfrac{(n-1)^2}4,\infty)$. The existence of
a point eigenvalue is the precise reason that our main theorem requires the condition $\delta >(n-1)/2$.
Writing $\alpha_1=s_1(n-1-s_1)$,
any positive number $$0<\bold s_\G<\delta-s_1$$ 
will be referred to as a {\it spectral gap} of $\G$.

Let $\mathcal C$ denote a Casimir element of $\op{Lie}(G)_\c$, which we normalize so that it
acts on $K$-invariant
smooth functions as the negative Laplacian $-\Delta$. 
Then $L^2(\G\ba G)$
contains the unique irreducible infinite dimensional subrepresentation $V$ (a complementary series representation) on which
 $\mathcal C$  acts by the scalar $\delta (\delta-n+1)$. 

Let $\hat K$ denote the unitary dual of $K$, that is,
the equivalence classes of all irreducible unitary representations of $K$.
For $n=2$,  $\hat K$ can be parametrized by $\z$ so that
$\ell\in \hat K$ corresponds to the one-dimensional space $V_\ell$ on which
$k_\theta=\begin{pmatrix} \cos \theta &\sin \theta \\ -\sin \theta&\cos\theta\end{pmatrix}$
 acts by $e^{2\ell i\theta}$. For $n=3$,
 $\hat K$ can be parametrized by $\z_{\ge 0}$ so that
$\ell\in \hat K$ corresponds to the irreducible $2\ell +1$ dimensional representation $V_\ell$.

As a $K$-representation,
 $V$ is decomposed into the orthogonal
sum $\oplus_{\ell\in \hat K}V_{\ell}$  with the subspace $V_{\ell}^M$ of $M$-invariant vectors being one dimensional.
Let $\phi_\ell\in C^\infty(\G\ba G)\cap L^2(\G\ba G)$ be a unit vector in $V_\ell^M$ for each $\ell\in \hat K$.
 We show that there exists $c_n(\ell )\ne 0$ such that for all $y>0$,
$$ \int_{n_x\in (N\cap \Gamma)\ba N} \phi_\ell (n_xa_y)\; dx =c_n({\ell})\cdot  y^{n-1-\delta}.$$ 

%We show that $V_\ell^M$ is spanned by
%a {\it real-valued} function, say, $\phi_\ell\in C^\infty(\G\ba G)$ of unit $L^2$-norm.
%{\footnote{ $C^\infty(\G\ba G)^M$ meansthe set of all $M$-invariant smooth functions on $\G\ba G$.}}

The inner product $\la \psi_1, \psi_2 \ra$ in $L^2(\G\ba G)$ is given by
$$\la \psi_1, \psi_2 \ra=\int_{\G\ba G}\psi_1(g)\overline{\psi_2(g)} \;dg$$ where $dg$ denotes a $G$-invariant measure on $\G\ba G$.

\medskip

The following is our main theorem on the effective equidistribution:
\begin{thm}\label{main} Let $n=2$ or $3$. Let $(n-1)/2<\delta<n-1$.
 For any $\psi\in C_c^\infty\left(\Gamma\bsl G\right)^M$, as $y\to 0$,
\begin{multline*} \int_{(N\cap \Gamma)\ba N} \psi(n_x a_y ) \, dx \\ =
  \sum_{\ell\in \hat K} c_n(\ell) \cdot \la   \psi ,\phi_\ell\ra \cdot
y^{n-1-\delta}   +O(\mathcal S_{2n-1}(\psi)\cdot y^{(n-1-\delta)+\tfrac{2{\bf s}_\G }{2n+1}}) \end{multline*}
where $\mathcal S_{2n-1}(\psi)$ denotes the $L^2$-Sobolev norm of $\psi$ of order $2n-1$.
 Moreover 
$$c_n(\ell) =O((|\ell|+1)^{(n-2)/2})\quad \text{and}\quad \sum_{\ell\in \hat K}|c_n(\ell) \la   \psi ,\phi_\ell\ra| =O( \mathcal S_2(\psi)).$$
\end{thm}

\begin{rmk}\rm
\begin{enumerate}
\item When $\G$ is a lattice in $G$, i.e., when $\delta=n-1$, an effective equidistribution
for expanding closed horospheres is well known, via the mixing of the geodesic flow
 and the thickening argument. This argument goes back to the 1970 thesis of Margulis \cite{Mt}
and was generalized by Eskin and McMullen \cite{EM}. For $n=2$, Sarnak \cite{Sa} obtained a sharper result, based on the study
of Eisenstein series. 
\item In principle, our methods should extend to prove
an analogous result for $G=\op{Isom}^+(\bH^n)$ for any $n\ge 2$;
however computations needed to understand $\phi_\ell$'s seem very intricate 
as the dimension gets higher.

\item 
When $\psi$ is {\it $K$-invariant},  Theorem \ref{main}
 was obtained in \cite{KO}. See also \cite{Kim} for its extensions to other rank one Lie groups.
%We do not know
%an extension of Theorem \ref{main} for the unit tangent bundle of an $n$-dimensional geometrically finite manifold for %$n\ge 4$.
\end{enumerate} \end{rmk}
 
As mentioned before,
our approach in proving Theorem \ref{main} is based on the existence of $L^2$-eigenfunctions on $\G\ba\bH^n$ for $\delta >(n-1)/2$
and hence cannot be applied to $\G$ with $\delta \le (n-1)/2$. 
However a non-effective version of Theorem \ref{main} is available for any $\delta>0$;
this was obtained in \cite{Ro} when 
$(N\cap \G)\ba N$ is compact and
in \cite{OS} in general.
In these papers, the coefficient of the main term was given in terms of the Burger-Roblin measure
associated to the stable horospherical foliation.
In applications to counting problems, it is much handier to have this coefficient in terms of a measure instead of an infinite sum.
For this reason, we present an alternative formulation of Theorem \ref{main} in Theorem \ref{main2}.

\medskip

\subsection{Equidistribution in ergodic terms}
Let $\nu_j$ denote the Patterson-Sullivan measure
on the limit set $\Lambda(\G)$ associated to the basepoint $j=(0_{n-1}, 1)\in \bH^n$, 
which is unique up to a constant multiple.

Sullivan gave an explicit construction of the base eigenfunction $\phi_0\in L^2(\G\ba G)^K$ using $\nu_j$:
\begin{equation}\label{ppp} \phi_0(n_xa_y)=\int_{u\in \br^{n-1}}
\left(\frac{(|u|^2+1)y}{|x-u|^2+y^2}\right)^\delta \; d\nu_j(u).\end{equation}
Here and also later, we identify $\c=\br^2$ for $n=3$, so that $|x-u|^2={(x_1-u_1)^2 +(x_2-u_2)^2}$ for $x=x_1+ix_2$ and $u=(u_1,u_2)$.
We normalize $\nu_j$ so that $\|\phi_0\|_2=1$ \cite{Sullivan1984}.

%Let $m^{\BMS}$ denote the Bowen-Margulis-Sullivan measure on $\G\ba G/M$ associated to $\{\nu_j\}$
%(see section \ref{bms} for the definition). As $\G$ is geometrically finite, $|m^{\BMS}|<\infty$ by Sullivan.
%We normalize the Patterson-Sullivan measure $\nu_j$ so that $\phi_0$ in this expression has $L^2$-norm one:
%$||\phi_0||_2 = 1.$ We keep this normalization throughout the paper and
%all the computations of constants are made with respect to this normalization.

Define the measure $\tilde m^{\BR}_{N}$ on $ G$ in the Iwasawa coordinates $G=KAN$:
for $\psi\in C_c( G)$,
$$\tilde m^{\BR}_N(\psi)=
\int_{KAN}\psi(k a_yn_x)y^{\delta -1} dx dyd\nu_j(k(0)) .$$

%with $o$ being the origin in the boundary $\partial(\bH^2)=\br\cup\{\infty\}$.
This measure is left $\G$-invariant and right $N$-invariant,
and the Burger-Roblin measure $m^{\BR}_{N}$ (associated to the stable horospherical subgroup $N$) 
is the measure on $\G\ba G$ induced from $\tilde m^{\BR}_N$.
 The BR measure $m^{\BR}_N$ %coincides with  a $G$-invariant finite measure if $\delta=1$ but%
is an infinite measure whenever $0<\delta<n-1$ \cite{OS}.

\begin{thm}\label{main2} Let $n=2$ or $3$ and $(n-1)/2 <\delta\le n-1 $.
For any $\psi\in C_c^\infty\left(\Gamma\ba G\right)^M$, as $y\to 0$,
\begin{multline*} \int_{(N\cap \Gamma)\ba N} \psi(n_x a_y ) \, dx  =
\kappa_\G \cdot m_N^{\BR}(\psi)\cdot y^{n-1-\delta} 
 \\ +O(\mathcal S_{2n-1}(\psi)\cdot y^{(n-1-\delta)+\tfrac{2{\bf s_\G}}{2n+1}}) \end{multline*}
where $\kappa_\G= \int_{x\in \br^{n-1}} {(1+|x|^2)^{-\delta}} {dx}\cdot \int_{\textsc{$n_x\in (N\cap \G)\ba N$}} {(1+|x|^2)^{\delta}} {d\nu_j(x)}$.
\end{thm}

%In order to extend the general strategy of \cite{KO} to non $K$-invariant functions,
%we need to control the pointwise bounds of $\phi_\ell$'s which
%requires an explicit formula of $\phi_\ell$.
\medskip

\subsection{Effective orbital counting and Affine sieves in sectors} 
Let $Q$ be a quadratic form over $\Q$ of signature $(n,1)$  and $v_0\in \z^{n+1}$ a non-zero vector
such that $Q(v_0)=0$. 
Let $G_0$ denote the identity component of $\SO_Q(\br)$. As well known, $G_0$ is isomorphic to
$G=\PSL_2(\br)$ (for $n=2$) and $G=\PSL_2(\c)$ (for $n=3$) as real Lie groups. Let $\G< G_0(\z)$ be a geometrically finite subgroup
with $\delta>(n-1)/2$ such that $v_0\G$ is discrete.
 For each square-free integer $d$, 
let $\G_d$ be a subgroup of $\G$
 containing $\{\gamma\in \G: \gamma\equiv e \text{ (mod d)}\}$ 
and satisfying $\text{Stab}_{\G}{v_0}=\text{Stab}_{\G_d}{v_0}$.

By a theorem of Bourgain, Gamburd and Sarnak \cite{BGS},
there exists a uniform spectral gap, say ${\s}_0>0$, for all $\G_d$, $d$ square-free. 

\medskip

Consider the representation $G\to G_0$
such that $N$ is contained in $\text{Stab}_G(v_0)$, and
fix a norm $\|\cdot \|$ on $\br^{n+1}$. For any subset $\Omega\subset K$ and $T>0$,
define the sector
$$S_T(\Omega):=\{v\in v_0A\Omega : \|v\|<T\} .$$

Define  $q_{\Omega}$ to be the maximum of
$0\le  q\le 1$ such that \begin{equation}\label{qo}
 \text{ $\nu_j(\e\text{-neighborhood of } \partial(\Omega^{-1}(0))) \ll \e^{q}$ for all small $\e>0$} .\end{equation}
 Note that if $\partial(\Omega^{-1}(0))\cap \Lambda(\G)=\emptyset$, then $q_\Omega=1$. 
We will say $\Omega$ admissible if $q_\Omega>0$.

\begin{thm}\label{ec2m}
Let $\Omega$ be an admissible left $M$-invariant Borel subset of $K$.
 Then for any $\gamma\in \G$, as $T\to \infty$, $$\#\{v\in v_0\Gamma_d\gamma \cap S_T(\Omega)\}=
\frac{ \Xi_{v_0}(\G, \Omega)}{{[\G:\G_d]}}  \cdot  T^\delta + O(T^{\delta-\tfrac{8{\s}_0}{n(n+9)(2n+1)q_\Omega}}).$$
\end{thm}

 Identifying $\G$ with its pull back in $G$, 
$\Xi_{v_0}(\G, \Omega)$ is given by \begin{equation}\label{xid} \Xi_{v_0}(\G, \Omega)= 
\kappa_\G
\int_{k\in \Omega^{-1}} \frac{ d\nu_j(k(0))} {\|v_0k^{-1} \|^{\delta}}  .\end{equation}
As $\nu_j$ is supported on the limit set $\Lambda(\G)$, $\Xi_{v_0}(\G, \Omega)>0$ if and only if
the interior of $\Omega^{-1}(0)$ intersects $\Lambda(\G)$.

 Given an integer-valued polynomial $F$ on the orbit $v_0\G$,
Theorem \ref{ec2m} has an application in studying integral points $\bf x$ lying in a fixed sector
with $F({\bf x})$ having at most $R$ prime factors (including multiplicities).
For instance, the following theorem can be deduced from Theorem \ref{ec2m} using the same analysis as in \cite[section 8]{KO}.

\begin{thm}\label{twin}   Suppose that $\Omega\subset K$ be an admissible subset such that the interior of
$\Omega^{-1}(0)$ intersects $\Lambda(\G)$.
Then there exists $R\ge 1$ (depending on $\s_0$) such that for each $1\le i\le n+1$,
 \begin{multline*} \#\{{\bf x}\in v_0\G \cap S_T(\Omega): \text{$x_1\cdots x_i$ has at most $R$ prime factors} \}
 \asymp \frac{T^\delta}{(\log T)^i}\end{multline*} 
where ${\bf x}=(x_1, \cdots, x_{n+1})$ and
$f(T)\asymp g(T)$ means that their ratio is between two positive constants uniformly for all $T\gg 1$.
\end{thm}

Theorem \ref{ec2m} is proved in \cite{OS} without an error term.
When the norm is {\it $K$-invariant} and $\Omega=K$, it was also proved in \cite{KO}. 
Theorem \ref{twin} for $\Omega=K$ has been obtained in \cite{KO} (also see \cite{KO1}).
%The same analysis works in deducing Theorem \ref{twin} from Theorem \ref{ec2m}.

\subsection{Organization:} Sections 2-4 are devoted to understanding the base eigenfunctions $\phi_\ell$'s
and their integrals over closed $N$-orbits.
In section \ref{rai}, we find  a {\it computable} recursive formula (Theorem \ref{zl2}) for
a raising operator among $M$-invariant vectors in a general complementary series representation of $G$.
Using this, in section 3, we obtain an explicit description of $\phi_\ell$'s
 which turn out to be related to the Legendre polynomials for $n=3$.
Understanding each $\phi_\ell$ as a function of $\G\ba G$, rather than as a vector in the
Hilbert space $L^2(\G\ba G)$, is crucial in our approach, as we need to deal with
several convergence issues of the integrals of $\phi_\ell$'s as well as to thicken the $N$-integrals
of $\phi_\ell$'s uniformly over all $\ell$'s.
In section 4, we compute the $N$-integrals of $\phi_\ell$'s and compute $c_n(\ell)$'s explicitly (modulo $c_0$).
In section 5, we carry out the thickening of the $N$-integrals of $\phi_\ell$'s uniformly. Since
$\phi_\ell$'s are not supported on compact subsets of $\G\ba G$, this step is delicate,
as we need to ensure that there is at most a polynomial error term in $\ell$ in this procedure.
The equidistribution theorems \ref{main} and \ref{main2} are proved in section 6 and 7 respectively.
In section 8, we deduce Theorem \ref{ec2m} and Theorem \ref{m} from Theorem \ref{main2}.

\medskip
Added in print: Soon after
 we submitted the first version of our paper to the arXive,
we received a preprint by Vinogradov \cite{V},
which also proves Theorem \ref{m} (with a weaker error term)
using different methods. 

\bigskip
\noindent{\bf Acknowledgment:} We thank Peter Sarnak for useful comments
on the preliminary version of this paper.

\section{Ladder operators}\label{rai}
\subsection{Notations and Preliminaries}
Let $G$ be $\PSL_2(\br)$ or $\PSL_2(\c)$. Hence as a real Lie group, $G$ is isomorphic to
the identity component of $\SO(n,1)$ for $n=2$ and $3$ respectively.
In this subsection, we introduce notations which will be used throughout the paper and review some basic facts
about representations of $G$.
Let $K$ be a maximal compact subgroup of $G$. Denoting by $\mathfrak g$ and $\mathfrak k$ the Lie algebras of $G$ and $K$
respectively,
let $\mathfrak g=\mathfrak k \oplus \mathfrak p$ be the corresponding Cartan decomposition of  $\mathfrak g$.
 Let $A=\exp (\mathfrak a)$ where
$\mathfrak a$ is a maximal abelian subspace of $\mathfrak p$
and let $M$ be the centralizer of $A$ in $K$.

Define the symmetric bi-linear form $\la \cdot, \cdot \ra$ on $\mathfrak g$ by
$$\la X, Y\ra :=\frac{1}{2(n-1)} B(X,Y)$$
 where $B(X,Y)=\op{Tr}(\text{ad}X \text{ad} Y)$ denotes the Killing form for $\mathfrak g$.
The reason for this normalization is so that the Riemmanian metric on $G/K$ induced by
 $\la \cdot, \cdot \ra$ has constant curvature $-1$.

 Let $\{X_i\}$ be a basis for $\mathfrak g_{\c}$ over $\c$;
put $g_{ij}=\la X_i, X_j\ra $ and let $g^{ij}$ be
the $(i,j)$ entry of the inverse matrix of $(g_{ij})$.
 The element
$$ \mathcal C =\sum g^{ij}X_iX_j$$
is called the Casimir element of $\mathfrak g_{\c}$ (with respect to $\la \cdot, \cdot \ra$).
 It is well-known that this definition is independent of the choice of a basis and that
$\mC$ lies in the center of the universal enveloping algebra $U(\mathfrak g_\c)$ of $\mathfrak g_{\c}$.

 %We compute the ladder operators between $V_\ell^M$'s. 

For $G=\PSL_2(\br)$,
set $$K=\op{PSO}(2)=
\left\{k_\theta=\begin{pmatrix} \cos \theta & \sin \theta \\ -\sin \theta &\cos\theta\end{pmatrix}: \theta\in [0,\pi)\right\};$$
  and for $G=\PSL_2(\c)$,
we set $K=\PSU(2)$. 
In both cases, we set $N$ to be the strict upper triangular subgroup of $G$ and $A$ the diagonal subgroup
consisting of positive diagonals.
 We have
the Iwasawa decomposition $G=NAK$: any element $g$ of $G$ is written uniquely
as $g=n_xa_yk$ where $n_x=\begin{pmatrix} 1 & x \\ 0& 1\end{pmatrix}\in N$,
 $a_y=\begin{pmatrix} \sqrt y & 0 \\ 0& \sqrt y^{-1} \end{pmatrix}$ and $k\in K$.
 Note that
 $M=\{e\}$ 
and $ M=\left\{\begin{pmatrix} e^{i\theta} & 0\\0& e^{-i\theta}\end{pmatrix}:  \theta\in [0,\pi)\right\}$ respectively.

Set $\bH^n=\{(x, y): x\in \br^{n-1}, y>0\}$ and $j=(0_{n-1},1)$. The group $G$ acts on $\bH^n$
via the extension of the M\"obius transformation action on the boundary $\br^{n-1}\cup\{\infty\}$.
Under this action, we have $K=\op{Stab}_G (j)$ and $\bH^n=G/K \simeq \exp(\mathfrak p)$.
Now the Laplacian operator $\Delta$ on $\bH^n$ is respectively  given by
$$\Delta=-y^2\left(\frac{\partial^2}{\partial x^2} +\frac{\partial^2}{\partial y^2}
\right)  \;\;\text{and}\;\; \Delta=-y^2\left(\frac{\partial^2}{\partial x_1^2}
+ \frac{\partial^2}{\partial x_2^2} +\frac{\partial^2}{\partial y^2} 
\right) + y \frac{\partial}{\partial y}$$
according as $n=2$ or $3$.
By Kuga's lemma, 
 for all $\psi\in C^\infty( G)^K=C^\infty(\bH^n)$, we have
 $$\mathcal C(\psi)=-\Delta(\psi).$$

Consider the following elements of $\mathfrak g$: $$H=\bpm 1 & 0 \\ 0 & -1\ebpm, \hskip 10pt E 
	= \bpm 0 & 1\\ 0 & 0 \ebpm \hskip 5pt \text{ and } \hskip 5pt F = \bpm 0 & 0 \\ 1 & 0 \ebpm.$$
 For $G=\PSL_2(\br)$,
since $\{H,E, F\}$ is a basis for $\mathfrak g_\c$, we
can compute $\mC$ by a direct method and obtain: 
\begin{lem}
For $n=2$, we have
$$\mC=\tfrac{1}{4} H^2 + \tfrac{1}{2}(EF+FE).$$
\end{lem} We set $\mC_K =E-F=\tfrac{\partial}{\partial\theta}$.

To compute the Casimir element for $n=3$, note that the Lie algebra of $K=\PSU(2)$ is 
	$$\mathfrak k=\left\{X\in \op{ M}_2(\C)\;:\; \bar X^t = -X, \;\; {\rm tr}(X)=0\right\}.$$ 
The elements $\tfrac{i}{2}H, X_1:=\tfrac{1}{2}(E-F), X_2:=\tfrac{i}{2}(E+F)$ generate $\mathfrak k$
as a real vector space.
Set 
\begin{equation}\label{eplus} D:=\tfrac{i}{2}H  \in \mathfrak k_\c, \;\; E^+:=X_1-i X_2\in \mathfrak k_\c\;\;
\text{and}\;\; E^-:= -X_1-i X_2\in \mathfrak k_\c .\end{equation}

The elements $\tfrac{1}{2}H$, $Y_1:=\tfrac{1}{2} (E+F)$ and $Y_2:=\tfrac{i}{2}(E-F)$
form a basis of $\mathfrak p$ over $\br$. 
Set
\begin{equation}\label{rl}\tilde H := \tfrac{1}{2}H \in \mathfrak p_\c,\;\; R := Y_1-iY_2\in \mathfrak p_\c
\;\;
\text{and}\;\;  L := -Y_1 -iY_2\in \mathfrak p_\c .\end{equation}

\begin{lem} 
For $n=3$, we have
$$ \mC=\mC_K + \tilde H^2 -\tfrac{1}{2}( RL+LR ) $$
where $\mC_K :=-D^2 +\tfrac{1}{2}(E^+E^-+E^-E^+)$ is the Casimir element of $\mathfrak k_\c$
(up to a constant multiple).
\end{lem}
\begin{proof} Note that $\{D, E^+, E^-, \tilde H, R, L\}$ forms a basis of $\mathfrak g_{\c}$.
 We check  $[D, E^{\pm}]=\pm iE^{\pm}$, $[D,R]-iR$, $[D,L]=-iL$, $[E^+, E^-]=-2iD$,
$[E^+,R]=[E^-,L]=[D, \tilde H]=0$, $[E^-, \tilde H]=-L, [E^+,\tilde H]=-R$,
$[E^-,R]=-2\tilde H$, $[E^+,L]=-2\tilde H$, $[R,\tilde H]=-E^+$,
$[L, \tilde H]=E^-$, and $[R,L]=2iD$.
 
Using these relations, we can compute the matrix $g_{ij}$ used in the definition of $\mC$ and
obtain the above formula for $\mC$.
Since $\{D, E^+, E^-\}$ forms a basis of $\mathfrak k_{\c}$,
we compute that $-D^2 +\tfrac{1}{2}(E^+E^-+E^-E^+)$  is the Casimir element of $\mathfrak k_\c$
(up to a scalar multiple) using the above relations.
\end{proof}

\subsection{Complementary series Representation of $G$} Let $V$ be an infinite dimensional irreducible unitary representation of $G$.
 Denote by $V^K$ and $V^M$ the subspaces of $K$-invariant and $M$-invariant vectors respectively.
We assume that $V^K$ is non-trivial and fix a unit vector in $v_0\in V^K$, which
is unique up to a scalar multiple.

Let $V^\infty$ denote the set of smooth vectors of $V$, i.e., $v\in V^\infty$
if the map $g\mapsto gv$ is a smooth function  $G\to V$.
Every element of $\mathfrak g$ acts as a differential operator on $V^\infty$:
for $X\in \mathfrak g$ and $v\in V^\infty$,
$$\pi(X)(v):=\left.\frac{d}{dt} (\exp (tX). v)\right|_{t=0}$$
where $\exp(X)=\sum_{j=0}^\infty \frac{X^j}{j!}$ denotes the usual exponential map $\frak g\to G$.
This action extends to the action of the universal enveloping algebra $U(\mathfrak g_\c)$ on $V^\infty$.

Denote by $\hat G$ the unitary dual of $G$. A representation $\pi\in \hat G$ is called {\it tempered}
if for any $K$-finite vectors $v, w$ of $\pi$, the matrix coefficient function
$g\mapsto \la \pi(g)v, w\ra$ belongs to $L^{2+\e}(G)$ for any $\e>0$.
It follows from the classification of $\hat G$ (cf. \cite[Thm. 16.2-3]{Knapp}) that
the non-tempered spectrum of $\hat G$ consists of the trivial representation and
the complementary series representation $(\pi_s,\mathcal H_s)$ parametrized by $\tfrac{n-1}{2}<s <(n-1)$, where
$\mC$ acts on $\mathcal H_s^\infty$ by the scalar $s(s-n+1)$.

We fix $V=\mathcal H_s$ for $\tfrac{n-1}{2}<s <(n-1)$.
If $n=2$, the unitary dual $\hat K$ can be parametrized by $\z$ so that
$\ell\in \hat K$ corresponds to the one dimensional representation $V_\ell$ on which $\mathcal C_K$ acts
by the scalar $-4\ell^2$. For $n=3$,
 $\hat K$ can be parametrized by $\z_{\ge 0}$ so that
$\ell\in \hat K$ corresponds to the irreducible $2\ell +1$ dimensional representation $V_\ell$.

\begin{lem}
  For each $\ell\ge 0$, $\mathcal C_K$ acts on $V_\ell$
as the scalar $\ell (\ell +1)$.
\end{lem}
 \begin{proof} If $w_\ell$ is the highest weight vector
of $V_\ell$, then $D(w_\ell)=i\ell w_\ell$. Using $[E^+, E^-]=-2iD$,
we can write $\mC_K=-D^2 -iD +E^-E^+$. Hence $\mC_K(w_\ell)= \ell^2 w_\ell +\ell w_\ell$ since $E^+(w_\ell)=0$.
Since $\mC_K$ acts by the scalar on $V_\ell$ as $V_\ell$ is irreducible, the claim follows.
\end{proof}

As a $K$-representation, we write
 $$V=\oplus_{\ell \in \hat K} m_\ell  V_{\ell}$$ where the multiplicity $m_\ell$ of $V_\ell$
 is at most one for each $\ell$ (see the remark following Theorem 4.5 of \cite{Wa}); in fact we show in the next subsection that $m_\ell=1$ for all
$\ell\in \hat K$.
We also have that the space $V_\ell^M$ is at most
 one dimensional \cite{Di}. 

\subsection{Ladder operators}
In this subsection, we will compute the Ladder operators which maps $V_\ell^M$ to $V_{\ell+1}^M$ and
 use them to obtain an explicit recursive formula for a unit vector of $V_\ell^M$ starting from $v_0$.
These are well-known for $n=2$.
\subsubsection{The case $G=\PSL_2(\br)$}
Set
 $$\mathcal R=\frac{1}{2} (\pi(H)+ i\pi( E +F)) \quad \text{and}\quad \overline{\mathcal R}=\frac{1}{2} (\pi(H)-i\pi(E+F)) .$$
These are called raising and lowering operators respectively and it is well-known that
 $\mathcal R(V_\ell)=V_{\ell+1}$ and $\overline{\mathcal R}(V_\ell)=V_{\ell-1}$ for any $\ell\in \z$ (cf. \cite[Prop. 2.5.2]{Bump}.
%In particular, $m_\ell =1$ for all $\ell \in \z$.
For a fixed unit vector $v_0\in V^K$, put
\begin{equation*} v_\ell=\begin{cases} \mathcal R^\ell(v_0) &\text{ if $\ell\ge 0$ }\\
             \overline{\mathcal R}^{|\ell|}(v_0) & \text{ if $\ell <0$ .}
 \end{cases}\end{equation*}

Then $V=\oplus_{\ell \in \hat K} \c v_\ell$ and
for $\ell \ge 0$, we have (see \cite{BKS}):
\be\label{bks} 
 \|v_{\pm \ell}\|_2= \frac{\sqrt{\Gamma(s+\ell)\Gamma(1-s+\ell)}}{\sqrt{\Gamma (s) \Gamma(1-s)}}.
\ee where $\G(x)$ denotes the Gamma function and $\|v\|$ denotes
 the norm of $v\in V$: $\|v\|=\sqrt{\langle v, v\rangle} $.

\subsubsection{The case $G=\PSL_2(\c)$} Recall the elements $D, E^{+},E^-,\tilde H, R,L$ from \eqref{eplus} and \eqref{rl}.
For each $\ell\ge 0$, the $K$-space $V_\ell$ is the irreducible representation
of $K$ of dimension $2\ell +1$. The operators $E^{\pm}$ move between different $M$-types inside each fixed $V_\ell$
and $R$ (resp. $L$) maps the highest (resp. lowest) weight vector space of each $V_\ell$ into the highest (resp. lowest)
 weight vector space of $V_{\ell+1}$.

We will show that the following differential operator $\mathcal Z_\ell$ maps $V_\ell^M$ to $V_{\ell+1}^M$: for each $\ell \ge 0$,
set  $${\mathcal Z}_\ell := \tfrac{1}{2}\left(RE^- +LE^+ -2(\ell+1) \tilde H\right).$$

\begin{lem}\label{ckc} \label{zh} For each $\ell\in \z_{\ge 0}$,
\begin{enumerate} 
 \item ${\mathcal C}_K {\mathcal Z}_\ell= {\mathcal Z}_\ell {\mathcal C}_K -i(RE^-- LE^+)D -
2\tilde H({\mathcal C}_K +D^2) +2(\ell+1)({\mathcal Z}_\ell+ \ell \tilde H );$
\item $D{\mathcal Z}_\ell  = {\mathcal Z}_\ell D$;
\item ${\mathcal Z}_\ell \tilde H = \tilde H {\mathcal Z}_\ell +\mathcal C-2\mathcal {\mathcal C}_K -D^2-\tilde H^2.$
\end{enumerate}
 \end{lem}
\begin{proof}
Since $\mathcal {\mathcal C}_K=-D^2 +\tfrac{1}{2}(E^+E^-+E^-E^+)$, we compute
	\begin{enumerate}
	\item ${\mathcal C}_K\tilde H = \tilde H {\mathcal C}_K -2 {\mathcal Z}_\ell -2\ell\tilde H$,
	\item ${\mathcal C}_KR = R{\mathcal C}_K -2iDR -2\tilde HE^+$, and
	\item ${\mathcal C}_KL = L{\mathcal C}_K +2iDL-2\tilde HE^-$.
	\end{enumerate}
These relations imply
	\begin{align*}
	&{\mathcal C}_K {\mathcal Z}_\ell = \tfrac{1}{2}\left(RE^-+LE^+ -2(\ell+1)\tilde H\right){\mathcal C}_K\\&
	 -\left\{iDRE^- -iDLE^++\tilde H (E^+E^-+E^-E^+)-2(\ell+1)({\mathcal Z}_\ell +\ell\tilde H)\right\}.
	\end{align*}
Using
 $E^+E^- = {\mathcal C}_K +D^2-iD$,
$E^-E^+={\mathcal C}_K +D^2+iD$, $DRE^-  = RE^-D$ and
$DLE^+ = LE^+D,$
we compute that $ {\mathcal C}_K {\mathcal Z}_\ell$ is equal to
	$$ {\mathcal Z}_\ell {\mathcal C}_K -i(RE^-D- LE^+D) -\tilde H(2{\mathcal C}_K +2D^2) +2(\ell+1){\mathcal Z}_\ell
 +2(\ell+1)\ell \tilde H.
	$$

For (2),  we note that
$[D, R] =iR $, $[D, L]=-iL  $, $[\tilde H, D] = 0$, $[D, E^\pm] = \pm iE^\pm$. Hence
	\begin{align*}
	D{\mathcal Z}_\ell &= \tfrac{1}{2}\left(DRE^- +DLE^+-2(\ell+1)D\tilde H\right)\\
	&=\tfrac{1}{2}\left\{(iR+RD)E^- +(-iL +LD)E^+ -2(\ell+1)\tilde HD\right\}\\
	%&=\tfrac{1}{2}\left\{iRE^- +R(-iE^- +E^-D) -iLE^++ L(iE^++E^+D) -2(\ell+1)\tilde HD\right\}\\
	&=\tfrac{1}{2}\left\{RE^-+LE^+ -2(\ell+1)\tilde H\right\}D = {\mathcal Z}_\ell D.
	\end{align*}
(3) can be proved similarly using $\mathcal C =\mathcal C_K+\tfrac{1}{2}(2\tilde H^2  -RL-LR)$.\end{proof}

\begin{prop}\label{zl1} For each $\ell\ge 0$,
 $${\mathcal Z}_\ell (V_\ell^M)\subset  V_{\ell+1}^M .$$
\end{prop}

\begin{proof} 
Note that $v\in V_\ell$ if and only if ${\mathcal C}_Kv = \ell(\ell+1)v$,
and $v\in V^M$ if and only if $Dv=0$. 

Let $v\in V_\ell^M$. 
Using $Dv=0$ and Lemma \ref{ckc} (1), we deduce that 
	\begin{align*}
	{\mathcal C}_K ({\mathcal Z}_\ell v) &=\ell(\ell+1){\mathcal Z}_\ell v
-2\ell(\ell+1)\tilde Hv +2(\ell+1){\mathcal Z}_\ell v +2\ell(\ell+1)\tilde Hv\\
	&=(\ell+1)(\ell+2){\mathcal Z}_\ell v.
	\end{align*}
Hence ${\mathcal Z}_\ell v \in V_{\ell+1}.$
By Lemma \ref{ckc} (2), we have
	$$D{\mathcal Z}_\ell v = {\mathcal Z}_\ell(Dv) = 0,$$
and hence ${\mathcal Z}_\ell v \in V_{\ell+1}^M.$
\end{proof}

Fixing a unit vector $v_0\in V^K$, 
define $v_\ell$, $\ell\ge 1$, recursively:
 \be\label{velldef} v_\ell:={\mathcal Z}_\ell(v_{\ell-1}).\ee
Put \be \label{abl} a_\ell := -2\ell+1\quad\text{and}\quad  b_\ell 
	:=(\ell-1)^2 (\ell(\ell-2)-s(s-2)) .\ee

\begin{thm}[Recursive formula for $v_\ell$]\label{zl2}
For $\ell \geq 1$, 
	$$v_\ell =  a_\ell \tilde H v_{\ell-1}+ b_\ell v_{\ell-2}$$
where $v_{-1}$ is understood as the zero vector.
%	$$\mathcal Cv = \lambda v$$
%	and
%	$${\mathcal C}_K v_\ell = \mu_\ell v_{\ell}\hskip 10pt\text{ for } \mu_\ell = -\ell(\ell+1).$$
\end{thm}
\begin{proof} 

%Note that $\mathcal C$ acts on $V$ by $\lambda=-s(n-1-s)$.
For $\ell\geq 0$ and $m\geq 0$, we have
	\begin{align*}
	{\mathcal Z}_\ell &= \tfrac{1}{2}(RE^-+LE^+-2(\ell+1)\tilde H)\\
	&= \tfrac{1}{2}(RE^-+LE^+-2(m+1)\tilde H+(-2(\ell+1)+2(m+1))\tilde H)\\
	&=\mathcal Z_m +(m-\ell)\tilde H
	\end{align*}
and hence
	\begin{equation} \label{zml} {\mathcal Z}_\ell v_m = v_{m+1}+(m-\ell)\tilde Hv_m.\end{equation}

 We need to show that for $\ell \geq 1$, 
	$$v_\ell = a_\ell \tilde H v_{\ell-1}+ b_\ell v_{\ell-2}.$$

Since $E^+v_0 = E^-v_0=0$, we have
$v_1=  {\mathcal Z}_0v_0 = -\tilde Hv_0 .$	
	For the induction process, assume
	$v_\ell = a_\ell \tilde H v_{\ell-1} +b_\ell v_{\ell-2}.$
	We deduce that, using Lemma \ref{zh} (3) and \eqref{zml},
	\begin{align*}
	v_{\ell+1} &= {\mathcal Z}_\ell v_\ell 
\\ &= a_\ell({\mathcal Z}_\ell \tilde H v_{\ell-1}) + b_\ell ({\mathcal Z}_\ell v_{\ell-2})\\
	& = a_\ell(\tilde H {\mathcal Z}_\ell +\mathcal C-2\mathcal {\mathcal C}_K -D^2-\tilde H^2)v_{\ell-1}
 +b_\ell(v_{\ell-1}-2\tilde Hv_{\ell-2})
.\end{align*}

Observe that $\mathcal C v_{\ell-1}=\lambda v_{\ell-1}$ with $\lambda=-s(2-s)$,
$\mathcal C_K v_{\ell-1}=\ell (\ell-1) v_{\ell-1}$ and $D v_{\ell-1}=0$.

Using the induction hypothesis, 
we deduce
\begin{align*}
	v_{\ell+1} &=a_\ell(\tilde H(v_\ell -\tilde Hv_{\ell-1})+(\lambda -2\ell({\ell-1}))v_{\ell-1} -\tilde H^2v_{\ell-1})\\
	& +b_\ell(v_{\ell-1}-2\tilde Hv_{\ell-2})\\
	&=a_\ell\tilde Hv_\ell -2\tilde H(a_\ell\tilde H v_{\ell-1} +b_\ell v_{\ell-2})+(a_\ell(\lambda-2\ell (\ell-1))+b_\ell)v_{\ell-1}\\
	&=(a_\ell-2)\tilde Hv_\ell +(a_\ell(\lambda-2\ell({\ell-1}))+b_\ell)v_{\ell-1}\\
	&= a_{\ell+1}  \tilde H v_{\ell}+ b_{\ell+1} v_{\ell-1}.\end{align*}
This finishes the proof. \end{proof}

For $x>0$, the Gamma function $\Gamma(x)$ is defined to be the integral $\int_0^\infty e^{-t} t^{-x-1} dt$:
it satisfies $\Gamma(x+1)=x\Gamma(x)$ and
for each positive integer $\ell$, $\Gamma (\ell)=\ell !$.

\begin{lem}\label{vbdd} 
For each $\ell \ge 0$, 
	$$\|v_\ell\| = \frac{\ell !}{\sqrt{2\ell+1}}\frac{\sqrt{\Gamma(s+\ell)\Gamma(2-s+\ell) }}
{\sqrt{\Gamma(s)\Gamma(2-s)}}.$$
In particular, $v_\ell \ne 0$ for each $\ell \ge 0$.
\end{lem}
\begin{proof}
Note that $v_\ell$'s are mutually orthogonal to each other. 
By Lemma \ref{zl2}, we have $v_\ell = a_\ell \tilde Hv_{\ell-1}+b_\ell v_{\ell-2}$. 
Therefore
	\begin{align*}
	\|v_\ell\|^2 &= \left<a_\ell \tilde H v_{\ell-1}+b_{\ell}v_{\ell-2}, v_{\ell}\right> = 
\left<a_{\ell}\tilde H v_{\ell-1}, v_{\ell}\right>\\
	&=-\left<a_\ell v_{\ell-1}, \tilde Hv_{\ell}\right> 
	= -\left<a_\ell v_{\ell-1}, \frac{1}{a_{\ell+1}}v_{\ell+1}-\frac{b_{\ell+1}}{a_{\ell+1}}v_{\ell-1}\right>\\
	&=\frac{a_\ell b_{\ell+1}}{a_{\ell+1}} \|v_{\ell-1}\|^2.
	\end{align*}
It follows that
	\begin{align*}
	\|v_\ell\|^2 &= \frac{a_\ell b_{\ell+1}}{a_{\ell+1}} \cdot \frac{a_{\ell-1}b_{\ell}}{a_{\ell}} 
	\cdot \cdots \cdot \frac{a_{\ell+1-j}b_{\ell+2-j}}{a_{\ell+2-j}}\cdots \frac{a_1b_2}{a_2}\|v_0\|^2\\
	&=\frac{a_1}{a_{\ell+1}} \cdot \prod_{j=1}^{\ell} b_{j+1}\\&
 = \frac{1}{2\ell+1}\prod_{j=1}^\ell j^2(j^2-1-s(s-2) )\\ &
=\frac{(\ell !)^2}{2\ell+1}\prod_{j=1}^\ell (j+s-1)(j-s+1).	\end{align*}
Therefore the claim follows by the well-known properties of the Gamma function.
\end{proof}

Lemma \ref{vbdd} shows  that $\mathcal Z_\ell$ maps a non-zero vector to a non-zero vector.
Therefore Theorem \ref{zl2} together with Lemma \ref{vbdd} implies the following:
\begin{cor}\label{zlf}
 For each $\ell\in \z_{\ge 0}$,
$$\mathcal Z_\ell(V_\ell^M)=V_{\ell+1}^M\quad\text{and}\quad V^M=\oplus_{\ell\in \hat K} \c v_\ell.$$
\end{cor}

%: MIN 

\section{Explicit formulas for base eigenfunctions $\phi_\ell$}\label{s:phl}
Let $G=\PSL_2(\br)$ or $\PSL_2(\c)$, so that $G=\text{Isom}^+(\bH^n)$ for $n=2,3$ respectively.
We keep the notations for $N,A,M$, $\mathcal C$ and $\Delta$, etc. from the  section \ref{rai}. In particular,
$\mathcal C$ satisfies $\mathcal C(\psi)=-\Delta(\psi)$
 for all $\psi\in C^\infty(\G\ba G)^K=C^\infty(\bH^n)$.

Let $\G<G$ be a geometrically finite discrete subgroup  with critical exponent $\tfrac{n-1}2 <\delta< n-1$.
Let $\nu_j=\nu_j(\G)$ denote the Patterson-Sullivan measure on the limit set $\Lambda(\G)$ 
with respect to the basepoint $j=(0_{n-1},1)\in \bH^n$.
Up to a scaling, $\nu_j$ is the weak-limit as $t\to \delta^+$
of the family of measures
$$\nu_{j}(t):=\frac{1}{\sum_{\gamma\in \G} e^{-td(j, \gamma j)}}
\sum_{\gamma\in\G} e^{-td(j, \gamma j)} \delta_{\gamma( j)}$$
where $\delta_{\gamma(j)}$ is the dirac measure at $\gamma(j)$.

%For $n=2,3$, the Laplacian operators $\Delta$ on $\bH^n$ are respectively  given by
%$$\Delta=-y^2\left(\frac{\partial^2}{\partial x^2} +\frac{\partial^2}{\partial y^2}
%\right)  \;\;\text{and}\;\; \Delta=-y^2\left(\frac{\partial^2}{\partial x_1^2}
%+ \frac{\partial^2}{\partial x_2^2} +\frac{\partial^2}{\partial y^2} 
%\right) + y \frac{\partial}{\partial y}.$$

We consider the Hilbert space $L^2(\G\ba G)$ where
the inner product $\la \psi_1, \psi_2 \ra$ is given by
$$\la \psi_1, \psi_2 \ra=\int_{\G\ba G}\psi_1(g)\overline{\psi_2(g)} \;dg$$ where $dg$ denotes 
a $G$-invariant measure on $\G\ba G$.
As $dg$ is an invariant measure,
the action of $G$ on $L^2(\G\ba G)$ by right translations gives rise to a unitary representation.

Since the complementary series representations exhaust the non-tempered spectrum of $\hat G$ (cf. \cite[Thm. 16. 2-3]{Knapp}),
 we deduce the following from \cite{Su} and \cite{LP}: recall the notation $\mathcal H_s$ of the complementary
series representation of $G$ on which $\mC$ acts by $s(s-n+1)$.
 
\begin{thm} \label{lp} There exist 
$$\alpha_0=\delta(\delta-n+1)>\alpha_1\ge \cdots \ge \alpha_k > -(n-1)^2/4$$ such that 
 $$L^2(\G\ba G)=\mathcal H_{s_0=\delta}\oplus \cdots\oplus  \mathcal H_{s_k} \oplus \mathcal W$$
where
$(n-1)/2< s_i<(n-1)$ is given by the equation $\alpha_i=s_i(s_i-n+1)$ and $\mathcal W$ lies
in the tempered spectrum of $\hat G$.
\end{thm}

We set $V:=\mathcal H_{\delta}$. The base eigenfunction $\phi_0\in V^K$ for the Laplacian $\Delta$
 can be explicitly written 
as the integral of the Poisson kernel against $\nu_j$ (\cite{Su}):
	$$\phi_0(n_xa_y) = \int_{u\in \br^n} 
	\hat{\phi}_u(x,y) \; d\nu_j(u) $$ where $\hat{\phi}_u(x, y) := \left(\frac{(|u|^2+1)y}{|x-u|^2+y^2}\right)^\delta.$ 
We normalize $\nu_j$ so that $\|\phi_0\|_2=1$.

For $n=2$, we set 
\begin{equation*} \psi_\ell=\begin{cases} \mathcal R^\ell(\phi_0) &\text{ if $\ell\ge 0$ }\\
             \overline{\mathcal R}^{|\ell|}(\phi_0) & \text{ if $\ell <0$}
 \end{cases}\end{equation*}

Sine $\overline{\mathcal R}$ is the complex conjugate of $\mathcal R$,
$\psi_{-\ell}=\overline{\psi_\ell}$. Therefore for $n=2$, it suffices to describe $\psi_\ell$ for $\ell\ge 0$.

For $n=3$, we define $\psi_\ell$ recursively:
$$\psi_0:=\phi_0,\quad \text{ and}\quad \psi_\ell={\mathcal Z}_\ell (\psi_{\ell-1})\quad \text{ for each $\ell \ge 1$.}$$

\begin{dfn}\rm Let $n=2,3$. For each $\ell \in \hat K$,
define the unit vector in $V_\ell^M$ by:
$$\phi_\ell:=\frac{\psi_\ell}{\|\psi_\ell\|_2}\in C^\infty(\G\ba G)^M\cap L^2(\G\ba G).$$
\end{dfn}

We have by Corollary \ref{zlf},
 $$ V^M=\oplus_{\ell\in \hat K} \c \phi_\ell .$$

The rest of this section is devoted to obtaining  pointwise bounds
for these base eigenfunctions $\phi_\ell$'s in terms of $\phi_0$.

\subsection{Base eigenfunctions for  $G=\PSL_2(\br)$.}

\begin{thm}\label{l:rasing-lowering}\label{ul}\label{ub} Let $G=\PSL_2(\br)$.
For $\ell \in \z_{\ge 0}$, 
$$\phi_{\ell}(n_xa_y) = 
\tfrac{\sqrt{\Gamma(1 -\delta)\Gamma(\ell+\delta)}}{\sqrt{\Gamma (\delta) 
\Gamma(\ell+1-\delta)}} \cdot \int_{\br} \hat{\phi}_u(x,y) \left(\tfrac{(x-u) - iy}{(x-u)+ iy}\right)^{\ell}  d\nu_j(u).$$

In particular,
 $$\left|\phi_{\pm \ell}(n_xa_y)\right|  \ll \phi_0(n_xa_y)$$
with implied constant independent of $\ell$.
\end{thm}

\begin{proof}
We have
$$K=\{k_\theta=\begin{pmatrix} \cos\theta &\sin \theta\\ -\sin\theta &\cos\theta \end{pmatrix}: \theta\in [0,\pi)\}.$$
The raising operator $\mathcal R$ in the Iwasawa coordinates $n_xa_yk_\theta$
can be written as $$\mathcal R
 = e^{2i\theta}\left(iy\tfrac{\partial}{\partial x} +y\tfrac{\partial}{\partial y} +\tfrac{1}{2i}\tfrac{\partial }
{\partial \theta}\right)$$
(see \cite{Bump}).
Since $\|\psi_\ell\|$ is given in \eqref{bks}, it suffices to show that
\begin{equation}\label{psi1}\psi_{\ell}(n_xa_yk_\theta)
 =\tfrac{ e^{ 2\ell i\theta}\Gamma(\delta+\ell)}{\Gamma(\delta)}\int_{\br}
 \hat{\phi}_u(x,y) \left(\tfrac{(x-u)- iy}{(x-u)+ iy}\right)^{\ell} \; d\nu_j(u).\end{equation}

	The case $\ell=0$ is clear.
To use an induction, we assume that  \eqref{psi1} holds for $\ell$.
We compute
	\begin{multline*}
	\left(iy\frac{\partial}{\partial x} + y\frac{\partial}{\partial y}\right)\left(\hat{\phi}_u(x, y)\cdot \left(\tfrac{(x-u)-iy}{(x-u)+iy}\right)^\ell\right)\\ =  
\delta\cdot \hat{\phi}_u(x, y)\cdot \left(\tfrac{(x-u)-iy}{(x-u)+iy}\right)^{\ell+1}  + \ell\cdot \hat{\phi}_u(x, y)\cdot \left(\tfrac{(x-u)-iy}{(x-u)+iy}\right)^{\ell}\cdot\left(\tfrac{-2iy}{(x-u)+iy}\right)
	\end{multline*}

and 
	\begin{multline*}
	\frac{1}{2i}\frac{\partial}{\partial\theta}\left(\mathcal R^\ell\phi_0\right)(n_xa_yk_\theta) \\
	= \tfrac{\Gamma(\delta+\ell)}{\Gamma(\delta)} \cdot \ell\cdot e^{2i\ell\theta}\int_\R
\hat{\phi}_u(x,y)\left(\tfrac{(x-u)-iy}{(x-u)+iy}\right)^\ell\; d\nu_j(u).
	\end{multline*}
Hence 
	\begin{align*}
	&\left(\mathcal R^{\ell+1}\phi_0\right)(n_xa_yk_\theta) 
	= \tfrac{e^{2i(\ell+1)\theta} \Gamma(\delta+\ell+1)}{\Gamma(\delta)}\int_{\br}
\hat{\phi}_u(x, y)  \left(\tfrac{(x-u)-iy}{(x-u)+iy}\right)^{\ell +1} \; d\nu_j(u)\;.
	\end{align*} Hence \eqref{psi1} holds for $\ell+1$, finishing the proof.

\end{proof}

\subsection{Base eigenfunctions for  $G=\PSL_2(\c)$.}
The case of $n=3$ involves more complicated computations and 
it turns out that $\phi_\ell$'s are not uniformly bounded by $\phi_0$, but grow
polynomially as $\ell \to \infty$ (see Theorem \ref{bdd1}).

We parametrize elements of $K=\PSU(2)$ as 
$$K = \left\{k_{\mu_1, \mu_2, \theta}:
{0\leq \theta < \frac{\pi}{2}, \; 0\leq \mu_1 < \pi, \; 0\leq \mu_2 <2\pi, \atop \mu_1=0\text{ if } \theta=\frac{\pi}{2}\; 
	\text{ and }\;  \mu_2=0\text{ if } \theta=0}\right\}$$ 
where $$ k_{\mu_1, \mu_2, \theta}:=
	\bpm  e^{i\mu_1} \cos\theta &  e^{i\mu_2} \sin\theta\\  -e^{-i\mu_2} \sin\theta
&  e^{-i\mu_1}\cos\theta\ebpm .	$$

Using the coordinates $(x_1,x_2, y, \mu_1,\mu_2, \theta)$, 
the element $\tilde H$ defined in \eqref{rl} is given by
\begin{multline}\label{ht}
	\tilde H= -\cos(\mu_1+\mu_2)\sin(2\theta) \; y\tfrac{\partial}{\partial x_1} -\sin(\mu_1+\mu_2)\sin(2\theta) 
\; y\tfrac{\partial}{\partial x_2}\\ + \cos(2\theta)\; y\tfrac{\partial}{\partial y} +\tfrac{\sin(2\theta)}{2}\;
\tfrac{\partial}{\partial \theta}.
	\end{multline}

Recall 
$$\hat{\phi}_u(x, y) = \left(\frac{(|u|^2+1)y}{|x-u|^2+y^2}\right)^\delta.$$

Setting $\hat\phi_u(n_xa_yk)=\hat\phi_u(x,y,k):=\hat\phi_u(x,y)$, we may regard $\hat \phi_u$ as a function on $G$ and
 define
functions $\hat{\phi}_u^{(\ell)}$ on $G$ by the recursive formula:
 for each $\ell \geq 1$, 
	$$\hat{\phi}_u^{(\ell)} :=  (-2\ell+1 )  \tilde H(\hat{\phi}_u^{(\ell-1)})
 +(\ell-1)^2(\delta(2-\delta)+\ell(\ell-2))\hat{\phi}_u^{(\ell-2)}$$
	(as before, the terms which are not defined are understood as $0$). 
\begin{lem}\label{eqqq} For each $\ell\ge 0$, and $n_xa_y\in NA$,
we have $$\psi_\ell(n_xa_y) =\int_{u\in \br^2}\hat{\phi}_u^{(\ell)}(x,y, e) d\nu_j(u).$$
\end{lem}
\begin{proof}
By the recursive formula for $\mathcal Z_\ell$ given in Theorem \ref{zl2},
 for $\ell\geq 1$, we have
$$\psi_{\ell+1} = \mathcal Z_\ell(\psi_{\ell-1})=a_\ell \tilde H (\psi_\ell)+ b_\ell\psi_{\ell-1}$$
where 
$ a_\ell = -2\ell+1 $ and $b_\ell = 
(\ell-1)^2(\delta(2-\delta)+\ell(\ell-2))$.

 We prove the claim by induction. The case of $\ell=0$ is by definition of $\psi_0$. Suppose the claim holds
for all $j\le \ell -1$.
Observe that all the terms in \eqref{ht} for $\tilde H$ except for the term $y \frac{\partial}{\partial y}$ have 
$\sin(2\theta)$ and hence vanish when $\theta=0$. Therefore
\begin{align*} \psi_\ell(n_xa_y) 
&= a_\ell y \tfrac{\partial \psi_{\ell-1}}{\partial y} (n_xa_y)+ b_\ell \psi_{\ell-2}(n_xa_y)\\ &=
 a_\ell y \tfrac{\partial }{\partial y}\int_{u\in \br^2}\hat \phi_u^{(\ell-1)}(x,y) d\nu_j(u) + b_\ell \int_{u\in \br^2}\hat \phi_u^{(\ell-2)}(x,y) d\nu_j(u)
\\ &=
\int_{u\in \br^2}\left(  a_\ell  (\tilde H \hat \phi_u^{(\ell-1)})(x,y,e ) + b_\ell  \hat \phi_u^{(\ell-2)} (x,y,e )\right) d\nu_j(u) 
\\ &=\int_{u\in \br^2}\hat{\phi}_u^{(\ell)}(x,y, e) d\nu_j(u).
\end{align*}
\end{proof}

It turns out that the Legendre polynomials
appear in the formula for $\phi_\ell$:
let us denote by ${\bf P}_\ell(t)$
 the Legendre polynomial of degree $\ell$. It is
 defined by the recursive relation: ${\bf P}_0(t)=1$, ${\bf P}_1(t)=t$ and
\begin{equation}\label{leg} 
\ell \cdot {\mathbf P}_\ell(t)=(2\ell-1)t {{\bf P}}_{\ell-1}(t) -(\ell-1){{\bf P}}_{\ell-2}(t).\end{equation}

\begin{thm} \label{exp1} \label{bdd1} \label{Leee} \label{cr}  Let $\ell \ge 0$. For any $n_xa_y\in NA$, we have
 \begin{multline*}                                                                         
\phi_\ell(n_xa_y)={\sqrt{ 2\ell+1}}\cdot \tfrac{\sqrt{\Gamma(2 -\delta)\Gamma(\ell+\delta)}}{\sqrt{\Gamma (\delta) 
\Gamma(\ell+2-\delta)}}\cdot
 \int_{u\in \br^2}\hat{\phi}_u(x,y)\cdot
{\bf{P}}_\ell \left( \tfrac{y^2-|x-u|^2}{y^2+|x-u|^2}
  \right) \, d\nu_j(u). \end{multline*}
In particular,
$$\left|\phi_{\ell}(n_xa_y)\right| \ll \sqrt{\ell+1}\cdot  \phi_0(n_xa_y)$$
with the implied constant independent of $\ell$.
\end{thm}
\begin{proof} 
 Since $\|\psi_\ell\|$'s have been computed in Lemma \ref{vbdd}, it suffices to show that
\begin{equation}\label{plll}                                                                         
\psi_\ell(n_xa_y)=\ell !
\tfrac{\Gamma(\delta+\ell)}{\Gamma(\delta)}
 \int_{u\in \br^2}
 \hat{\phi}_u(x,y)\cdot {\bf{P}}_\ell \left( \tfrac{y^2-|x-u|^2}{y^2+|x-u|^2}
  \right) d\nu_j(u) .\end{equation}
 Since $\psi_\ell(n_xa_y) =\int_{u\in \br^2}\hat{\phi}_u^{(\ell)}(x,y, e) d\nu_j(u)$ by Lemma \ref{eqqq},
 \eqref{plll} follows if we show:
	\begin{equation}\label{show}
\hat{\phi}_u^{(\ell)}(x, y, e) = 
 \ell! \tfrac{\Gamma(\delta+\ell)}{\Gamma(\delta)}
\cdot \hat{\phi}_u (x,y)\cdot{\bf{P}}_\ell \left(\tfrac{ y^2-|x-u|^2}{y^2+|x-u|^2}\right).\end{equation}
Put  $B(x, y,u) = \frac{ y^2-|x-u|^2}{y^2+|x-u|^2}.$	
Since 
	$$y\tfrac{\partial}{\partial y}\hat{\phi}_u(x, y) = -\delta\cdot \hat{\phi}_u(x, y)\cdot B(x, y,u)\;\; \text{and}\;\;
	y\tfrac{\partial}{\partial y} B = -B^2+1,$$
we have
	\begin{align*}&
	\hat{\phi}_u^{(1)}(x, y, e) = a_1\left(y\tfrac{\partial}{\partial y}\hat{\phi}_u\right)(x,y,e)\\
	&= -\delta \hat{\phi}_u(x, y)\cdot a_1B = \delta \hat{\phi}_u(x, y){\bf P}_1(B)
	\end{align*}
proving the claim for $\ell=1$.
For the induction, we assume that for all $j\le \ell-1$,
	$\hat{\phi}_u^{(j)}(x, y, e) = d_j \hat{\phi}_u(x, y) j!\cdot {\bf P}_j(B)$
 where $d_\ell:=\prod_{j=1}^\ell(\delta+j-1)$. Then
	\begin{align*}&
	\hat{\phi}_u^{(\ell)}(x, y, e)  = a_\ell \cdot\left(y\tfrac{\partial}{\partial y}\hat{\phi}_u^{(\ell-1)}\right)(x, y, e)
 +b_\ell\cdot \hat{\phi}_u^{(\ell-2)}(x, y, e)=\\
	&a_\ell d_{\ell-1} \cdot (\ell-1)!\cdot y\tfrac{\partial}{\partial y}\left(\hat{\phi}_u(x, y){\bf P}_{\ell-1}(B)\right) 
	+b_{\ell} d_{\ell-2}\cdot \hat{\phi}_u(x, y)\cdot (\ell-2)!\cdot {\bf P}_{\ell-2}(B).
	\end{align*}
Since $b_\ell = (\ell-1)^2(-\delta+\ell)(\delta+\ell-2)$, we have $$d_{\ell-2}b_{\ell} = (\ell-1)^2d_{\ell-1}(-\delta+\ell).$$
Therefore \begin{align*}
& \hat{\phi}_u^{(\ell)}(x, y, e) \\ 
&= d_{\ell-1} \left\{ -a_\ell \delta B (\ell-1)! {\bf P}_{\ell-1}(B)  - a_{\ell} (B^2-1)(\ell-1)!
{\bf P}_{\ell-1}'(B) \right.
\\& \left.+(\ell-1)^2(-\delta+\ell)(\ell-2)! {\bf P}_{\ell-2}(B) \right\} \hat{\phi}_u(x, y)
 \\&=d_{\ell-1} \left\{ -a_\ell \delta B(\ell-1)!{\bf P}_{\ell-1}(B) -\delta (\ell-1)^2(\ell-2)!{\bf P}_{\ell-2}(B) \right.\\
	&\left. - a_\ell (B^2-1)(\ell-1)!{\bf P}_{\ell-1}'(B) +\ell (\ell-1)^2(\ell-2)!
 {\bf P}_{\ell-2}(B) \right\} \hat{\phi}_u(x, y). \end{align*}
Since
	$$(\ell-1)(\ell-1)!{\bf P}_\ell(B) = -a_\ell(B^2-1)(\ell-1)!{\bf P}_{\ell-1}'(B)
 +(\ell-1)^2\ell (\ell-2)!{\bf P}_{\ell-2}(B),$$
we have
	\begin{align*}
	\hat{\phi}_u^{(\ell)}(x, y, e) &= d_{\ell-1}\hat{\phi}_u(x, y)\cdot (\delta+\ell-1)\ell!{\bf P}_\ell(B) \\
	&=d_\ell \cdot \hat{\phi}_u(x, y)\cdot \ell!{\bf P}_\ell(B).
	\end{align*}
This proves \eqref{show}, in view of
 the relation $\G(t+1)=t \Gamma(t)$ and $d_\ell=\frac{\Gamma(\delta+\ell)}{\Gamma(\delta)}$.
Hence the first claim of the theorem is proved.

 Since $|{\bf P}_\ell(t)|\le 1$  for $|t|\le 1$
(see p. 987 of \cite{GR}), the second claim follows from the first claim.
\end{proof}

We do not know whether for all
$n_xa_yk\in G$, $\left|\phi_\ell(n_xa_yk)\right|$ is uniformly bounded by $\ell^N \phi_0(n_xa_y)$ for some $N>0$. 
However for our purpose it suffices to prove the following bound:
\begin{thm}\label{bdd}
For each $\ell\geq 0$, there exists a constant $C_\ell >0$ such that for any $n_xa_yk\in NAK$,
	$$\left|\phi_\ell(n_xa_yk)\right| \le C_\ell \cdot \phi_0(n_xa_y).$$
\end{thm}
\begin{proof}
%Recall that any $g\in G$ is of the form $n_xa_y k_{\mu_1,\mu_2,\theta}$ uniquely.
We use notations from the proof of Theorem \ref{cr}.
We compute	$$y\tfrac{\partial}{\partial x_1} \hat{\phi}_u = \delta\cdot \hat{\phi}_u \cdot \tfrac{-2y(x_1-u_1)}{|x-u|^2+y^2}, 
	\;\;y\tfrac{\partial}{\partial x_2} \hat{\phi}_u = \delta\cdot\hat{\phi}_u\cdot\tfrac{-2y(x_2-u_2)}{|x-u|^2+y^2}$$
and
	$$y\tfrac{\partial}{\partial y} \hat{\phi}_u = \delta \cdot \hat{\phi}_u\cdot \tfrac{|x-u|^2 -y^2}{|x-u|^2+y^2}.$$
Let
	$$A_1(x, y, u) = \tfrac{-2y(x_1-u_1)}{|x-u|^2+y^2},\;\; A_2(x, y, u) = \tfrac{-2y(x_2-u_2)}{|x-u|^2+y^2}$$
	and
	$$B(x, y, u) = \tfrac{|x-u|^2 -y^2}{|x-u|^2+y^2}.$$
For $\mu=(\mu_1, \mu_2)$ and $0\le \theta<\pi$, define
	$$\Phi_1(\mu, \theta) := -\sin(2\theta)\cos(\mu_1+\mu_2),\text{ and }\Phi_2(\mu, \theta) := -\sin(2\theta)\sin(\mu_1+\mu_2)$$
	and
	$\Psi(\theta) := \cos(2\theta)$
so that by \eqref{ht},
	$$\tilde H = \Phi_1y\frac{\partial}{\partial x_1} +\Phi_2y\frac{\partial}{\partial x_2} +\Psi y\frac{\partial}{\partial y}+
\frac{\sin(2\theta)}{2} \frac{\partial}{\partial \theta}.$$
Hence
	$$\tilde H\hat{\phi}_u = \delta\cdot \hat{\phi}_u\cdot\left(\Phi_1 A_1 +\Phi_2 A_2 +\Psi B\right);$$
so
	$$\hat{\phi}_u^{(1)} = \hat{\phi}_u\cdot a_1\delta\cdot \left(\Phi_1 A_1 +\Phi_2 A_2 +\Psi B\right).$$
We compute
	$$
	y\tfrac{\partial }{\partial x_1} A_1 
	=A_1^2+ B-1,\;\;
	y\tfrac{\partial}{\partial x_2} A_1 
	=A_1A_2 +B-1,\;\;
	y\tfrac{\partial}{\partial y} A_1 
	=A_1B;
	$$ $$
	y\tfrac{\partial}{\partial x_1} A_2 
	=A_1A_2 + B-1,\;\;
	y\tfrac{\partial}{\partial x_2}A_2 
	=A_2^2+B-1,\;\;
	y\tfrac{\partial }{\partial y} A_2 
	=A_2B;$$ and
	$$y\tfrac{\partial }{\partial x_1}B = A_1(B-1),\;\;y\tfrac{\partial}{\partial x_2}B =A_2(B-1),\;\;
	y\tfrac{\partial}{\partial y}B = (B-1)(B+1)
	.$$
We also compute:
	$$\tilde H(\Phi_1) = \Phi_1\Psi,\quad
	\tilde H(\Phi_2) = \Phi_2\Psi$$
	and $$\tilde H(\Psi) = -\Phi_1^2-\Phi_2^2 = -1+\Psi^2.$$

It follows that
	$\hat{\phi}_u^{(\ell)} = \hat{\phi}_u\cdot p_\ell(\Phi_1, \Phi_2, A_1, A_2, \Psi, B)$
where $p_\ell$ is a polynomial in $\Phi_1$, $\Phi_2$, $A_1$, $A_2$, $\Psi$ and $B$, whose coefficients are given
by monomials in
	$\left\{\pm 1, a_1, \ldots, a_{\ell}, b_1,\ldots, b_\ell, \delta\right\}.$
Since the absolute values of
$\Phi_1, \Phi_2, \Psi, A_1, A_2$ and $B$ are all bounded above by $1$,
 we deduce $$\left|\hat{\phi}_u^{(\ell)}(x, y, k) \right| \leq C_\ell \cdot \hat{\phi}_u(x, y).$$
for some constant $C_\ell >0$ independent of $x,y,k$.
Since $$\psi_\ell(n_xa_yk) =\int_{u\in \Lambda(\G)}\hat{\phi}_u^{(\ell)}(x,y,k) d\nu_j(u),$$
the claim follows.
\end{proof}

\section{Horospherical average of $\phi_{\ell}$}\label{s:h}
We let the notations $G=NAK, \G, \delta, \phi_0, \phi_\ell$, etc., be as in Section \ref{s:phl}.
We assume that $\G$ is geometrically finite with $(n-1)/2 <\delta <(n-1)$ and that
$\G\ba \G N$ is closed in $\G\ba G$ in the whole section.
The main goal of this section is to compute
the horospherical average of $\phi_\ell$ over $(\G\cap N)\ba N$ (Theorem \ref{ub2}).
Consider $\bH^n$ for $n=2$ or $3$ according as $G=\PSL_2(\br)$ or $\PSL_2(\c)$.
We let $X_0\in \T^1(\bH^n) $ be the upward normal vector based at $j=(0_{n-1},1)\in \bH^n$.
The map $g\mapsto g(X_0)$ induces the identification of
$G/M$ with $\T^1(\bH^n)$.
The horosphere in $\T^1(\bH^n)$ corresponding to $NM/M$ consists of the upward
 normal vectors on the
horizontal plane containing $j$, and hence based at $\infty\in \partial(\bH^n)$.

The assumption that $\G\ba \G N$ is closed is equivalent to saying that either
$\infty\notin \Lambda(\G)$ or $\infty$ is a parabolic fixed point of $\G$ \cite{Dalbo}.
Recall $\xi\in \Lambda(\G)$ is a parabolic fixed point if it is a unique fixed point in $\partial(\bH^n)$
of an element of $\G$. One of the important features of a geometrically finite group $\G$
is that any parabolic fixed point $\xi$ of $\G$ is bounded, meaning that the stabilizer
 $\op{Stab}_\G(\xi)$ acts cocompactly
on $\Lambda(\G)-\{\xi\}$ \cite{Bow}.

Therefore our assumption on the closedness of $\G\ba \G N$ implies the following:
\begin{lem}\label{bow}
$N\cap \G$ acts cocompactly on $\Lambda(\G)-\{\infty\}$.\end{lem}
As mentioned in the introduction, the rank of $\infty$ is the rank of $\G\cap N$
as a free abelian group.

\medskip
 
\begin{dfn}
	Given $\psi\in C(\Gamma\bsl G)^M$, define $\psi^N\in C(\Gamma\ba G)^M$ by
	$$\psi^N(g) := \int_{n_x\in (N\cap \Gamma)\bsl N} \psi(n_xg)\; dx\;$$
where $dx$ denotes the Lebesgue measure on $\br^n$, provided the integral converges.
\end{dfn}

\begin{prop}\label{co} There exists $c_n(0) >0$ such that for all $y>0$, 
 $$\phi_0^N (a_y) = c_n(0) \cdot y^{n-1-\delta}.$$
 \end{prop}
\begin{proof}
In \cite{KO},  it was shown that $\phi_0^N (a_y)$ converges absolutely and
that there are constants $c_n(0)>0$ and $d_n(0)\in \br$ such that for all $y>0$
	\begin{equation}\label{all} \phi_0^N (a_y) = c_n(0) y^{n-1-\delta}+ d_n(0) y^{\delta}.\end{equation}
Since $\phi_0 >0$ and \eqref{all} holds for all $y>0$, it follows that $d_n(0)\ge 0$. 
 We claim that $d_n(0)=0$.

When $\infty\notin \Lambda(\G)$, $\Lambda(\G)$ is a bounded subset of $\br^{n-1}$ and
we can show by direct computations:
\begin{equation}\phi_0^N(a_y)=\begin{cases}
              \frac{\sqrt{\pi} \G(\delta-\tfrac 12)}{\G(\delta)}
 \cdot \int_{u\in \Lambda(\G)} (|u|^2+1)^\delta \;d\nu_j (u)\cdot
  y^{1-\delta}            &\text{if $n=2$}  \\                    
\tfrac{\pi}{\delta-1} \cdot \int_{u\in \Lambda(\G)} (|u|^2+1)^\delta \;d\nu_j (u)\cdot
  y^{2-\delta} &\text{if $n=3$}
 \end{cases} \end{equation} (see \cite{KO}).

Now suppose $\infty\notin \Lambda(\G)$.
As $\G$ is geometrically finite, $\Gamma$ admits a polyhedron
fundamental domain $\mathcal F$ in $\mathbb H^n$ such that
$F_0 \times [Y_0,\infty)$ injects to $\mathcal F$ for some $Y_0\gg 1$ where $F_0$ is a fundamental domain
in $\br^{n-1}$ for $N\cap \G$.
Let $B_t=\{x\in F_0: |x|<t\}$ for $t>1$. 

When $\infty$ is a bounded parabolic fixed point of rank $n-1$, take $t_0$ so that
$B_{t_0}=F_0$, which is possible since $F_0$ is bounded in this case.
Then 
\be\label{zzz}{\int_{x\in B_{t_0}}\phi_0(n_x a_y)dx}\ge \tfrac{d_n(0)}{2} y^{\delta}.\ee
Therefore using the Cauchy-Schwartz inequality, we have \begin{align*} \|\phi_0\|_2^2 &\ge 
\int_{Y_0}^{\infty}\int_{B_{t_0}}\phi_0(n_xa_y)^2  y^{-n} dx dy\\ 
 &\ge  \tfrac{1}{\text{vol}(B_{t_0})} \int_{Y_0}^{\infty}
 \left( \int_{B_{t_0}}\phi_0(n_xa_y)dx\right)^2  y^{-n}  dy\\
&\ge    \tfrac{d_n(0)^2}{4\text{vol}(B_{t_0})} \int_{Y_0}^{\infty}  y^{2\delta-n}  dy.
\end{align*}
Since $\delta>(n-1)/2$, $\|\phi_0\|_2=\infty$ unless $d_n(0)\ne 0$.  Therefore $d_n(0)=0$.

The remaining case is when $n=3$ and $\infty$ is a bounded parabolic fixed point of rank one. In this case,
it was shown in the proof of \cite[Prop. 4.6]{KO} that for all sufficiently large $t\gg 1$, 
there exists $b_t\to 0$ as $t\to \infty$ such that for all $y>0$, $${\int_{x\in F_0- B_{t}}\phi_0(n_x a_y)dx}\le b_t y^{\delta}.$$
Therefore, if $d_3(0)$ were positive, then for some large $t_0>0$, we would have
\be\label{zzzz}{\int_{x\in B_{t_0}}\phi_0(n_x a_y)dx}\ge \tfrac{d_3(0)}{2} y^{\delta}.\ee
By repeating the same argument as in the previous case, this leads to a contradiction.
\end{proof}

\begin{lem}\label{zero}
 Let $y>0$ and $k\in K$.
For each $1\le i\le n-1$, we have
$$\int_{(N\cap \Gamma)\ba N} \tfrac{\partial}{\partial x_i} \phi_\ell(n_x a_yk)\; dx = 0.$$
\end{lem}

\begin{proof}  By Theorems \ref{ub} and \ref{bdd}, we have, for some $C_\ell >0$,
\be\label{bdd2} |\phi_\ell(n_x a_yk)|\le C_\ell \cdot \phi_0(n_x a_y).\ee

Suppose $n=3$.
If $\infty\notin \Lambda(\G)$ and hence $N\cap \G=\{e\}$,
 $|\phi_0(n_x a_yk)|\to 0$ as $|x|\to \infty$. Hence, by \eqref{bdd2}, as $|x|\to \infty$,
$$|\phi_\ell(n_xa_yk)|\to 0 .$$
 Therefore
\begin{align*} & \int_{x_1\in \br } 
\tfrac{\partial}{\partial x_1} \phi_\ell(n_x a_yk)\; dx_1  
\\&= \lim_{t\to\infty} \int_{-t}^{t} \tfrac{\partial}{\partial x_1} \phi_{ \ell}(n_x a_yk)\; dx_1 \\ &=
\lim_{t\to \infty} (  \phi_{ \ell}(n_{t+\sqrt{-1}x_2} a_yk)  -
 \phi_{ \ell}(n_{-t+\sqrt{-1} x_2} a_yk))=0.
\end{align*}
 The other case of $i=2$ is symmetric to this one.

When $\infty$ is a bounded parabolic fixed point of rank one,
we may assume without loss of generality that
$N\cap \G$ is generated by $n_{1}=\begin{pmatrix} 1 & 1\\ 0 & 1\end{pmatrix}$, so that there exists a
 fundamental domain $F_0$ in $\br^{2}$ inside $\{(x_1,x_2): 0\le x_1\le 1\}$.  In this case,
$\phi_0(n_x a_y)\to 0$
as $x_2\to \infty$ \cite{Sullivan1984}.
Hence
$|\phi_\ell(n_xa_yk)|\to 0$, as $x_2\to \infty$.
By a similar argument as above, this implies that $$
\int_{x_2\in \br } 
\tfrac{\partial}{\partial x_2} \phi_\ell(n_x a_yk)\; dx_2  =0 .$$
On the other hand, by the $n_1\in N\cap \G$-invariance of $\phi_\ell$,
$$\int_{x_1\in [0,1] } 
\tfrac{\partial}{\partial x_1} \phi_\ell(n_x a_yk)\; dx_1  =
  \phi_{ \ell}(n_{1} n_{\sqrt{-1} x_2} a_yk)  
 -\phi_{ \ell}(n_{\sqrt{-1} x_2} a_yk)=0 .$$

If $\infty$ has rank $2$,
we may assume that $N\cap \G$ is generated by $n_1$ and $n_{\sqrt{-1}}$. Then the
 the claim follows from Green's theorem and the invariance of $\phi_\ell$
by $N\cap \G$ as in the last argument. The case $n=2$ can be shown similarly.
\end{proof}

 For $\ell\ge 0$, define
 \begin{equation}\label{mlt} M_\ell(\theta):=\ell !\cdot {\bf P}_{\ell}(-\cos 2\theta). \end{equation}
Then  $M_0(\theta) = 1$ and the recursive relation \eqref{leg} for ${\bf P}_\ell$'s implies:
\begin{lem}\label{l_ml}\label{ml}
 \begin{enumerate}
  \item  For $\ell \ge 1$,
	$$M_\ell(\theta) = (-2\ell+1)\cos(2\theta)M_{\ell-1}(\theta)-(\ell-1)^2M_{\ell-2}(\theta).$$	
\item For each $\ell\ge 1$, \begin{multline*}
2(\ell+1) M_{\ell}(\theta)=  { 4(-2\ell+1)\cos(2\theta)} M_{\ell-1}(\theta) 
	\\ +{(-2\ell+1) \sin(2\theta)}  M_{\ell-1}'(\theta)+
{2(\ell-1)^2(\ell-2)}M_{\ell-2}(\theta).
\end{multline*}
 \end{enumerate}
\end{lem}

%We set $k_\theta=k_{0,0,\theta}$.
The following theorem implies in particular that the integral $\phi_{\ell}^N(a_y) $
converges, which is a priori unclear as $(N\cap \G)\ba N$ is not compact in general.
\begin{thm}\label{psi2}\label{ub2} 
	For $\ell\geq 0$, we have
	$$\phi_{\ell}^N(a_y) 
	= (-1)^{(n-2)\ell}  \cdot c_n(0)
\cdot \tfrac{\sqrt{\Gamma(n-1 -\delta)\Gamma(\ell+\delta)}}{\sqrt{\Gamma (\delta) 
\Gamma(\ell+n-1-\delta)}} \cdot   \sqrt{2(n-2)\ell+1} \cdot y^{n-1-\delta}.$$

In particular, 
 $$|\phi_\ell^N(a_y)| \ll \ell^{(n-2)/2} \cdot y^{n-1-\delta}$$
with the implied constant independent of $\ell \ge 1$.
\end{thm}

\begin{proof}
By  Theorems \ref{ub} and \ref{cr},  
	we have \begin{align*}
	\left|\int_{(N\cap \Gamma)\bsl N} \phi_{\ell} (n_xa_y)\; dx\right| &\leq
 \int_{(N\cap \Gamma)\bsl N} \left|\phi_{ \ell}(n_xa_y)\right|\; dx
\\	&\ll_{\ell}  \int_{(N\cap \Gamma)\bsl N} \phi_0(n_xa_y)\; dx .
	\end{align*}

Hence by Proposition \ref{co}, the integral $\phi_{\ell}^N (a_y)$ converges absolutely. 

Let $n=2$. 
By \eqref{bks}, it suffices to show
	\begin{equation} \label{ind1} \psi_\ell^{N} (a_yk_\theta)
 = e^{2 \ell i\theta} c_2(0)\tfrac{\Gamma(\ell+1-\delta)}{\Gamma(1-\delta)}y^{1-\delta} .\end{equation}
The case $\ell=0$ holds by Proposition \ref{co}.
To use an induction, we assume \eqref{ind1} is true for $\ell$.
 Then  applying Lemma \ref{zero},
	\begin{align*}
	\psi_{\ell+1}^{N}(a_yk_\theta) = &e^{2i\theta} \int_{(N\cap \Gamma)\bsl N}
 \left( iy \tfrac{\partial }{\partial x}+y\tfrac{\partial }{\partial y} \tfrac{1}{2i}\tfrac{\partial }{\partial \theta}\right)
\psi_{\ell}(n_x a_yk_\theta)\; dx\;
	\\&= e^{2i\theta} \cdot \left(y\tfrac{\partial }{\partial y} \tfrac{1}{2i}\tfrac{\partial}{\partial \theta}\right)
\psi_\ell^{N}(a_yk_\theta)
	\\& =e^{ 2i\theta}\cdot
\left(y\tfrac{\partial }{\partial y} \tfrac{1}{2i}\tfrac{\partial}{\partial \theta}\right)
\left( e^{2\ell i\theta} c_2(0) \tfrac{\Gamma(\ell +1-\delta)}{\Gamma(1-\delta)}y^{1-\delta} \right)
	\\& = e^{ 2(\ell+1)i\theta}  \cdot 
\left(c_2(0) \cdot \tfrac{ \Gamma(\ell +1-\delta)}{\Gamma (1-\delta)}\cdot \left((1-\delta) + \ell\right) y^{1-\delta} \right)
	\\& = e^{ 2(\ell+1)i\theta} \cdot \left(c_2(0) \cdot \tfrac{\Gamma (\ell +1+(1-\delta))}{\Gamma (1-\delta)} y^{1-\delta}  \right) \end{align*}
since $z\Gamma(z) =\Gamma(z+1)$.  This proves \eqref{ind1} for $\ell+1$.

Let $n=3$. 
Setting $q_\ell:=c_3(0)\prod_{j=1}^{\ell} (j+1-\delta)$, we claim 
that \begin{equation}\label{pz}\psi_{\ell}^N(a_yk) = q_\ell\cdot  y^{2-\delta}M_\ell(\theta)\end{equation}
where $M_\ell(\theta)$ is defined as in \eqref{mlt}.
Set $k= k_{\mu_1,\mu_2,\theta}$, $ a_\ell = -2\ell+1 $ and $b_\ell = 
(\ell-1)^2(\delta(2-\delta)+\ell(\ell-2))$ for simplicity.

Since $\phi_0$ is fixed by $K$, $E^{\pm}(\phi_0)=0$. Hence
	\begin{align*}
	\psi_1^{N}(a_yk) 
	=a_1\int_{(N\cap \Gamma)\bsl N} \tilde H(\phi_0)(na_yk)\; dn.
	\end{align*}

By Proposition \ref{co} and Lemma \ref{zero}, using $\frac{\partial}{\partial \theta}\phi_0(n_xa_yk)=0$ and \eqref{ht},
we have
	\begin{align*}
	\psi_1^{N}(a_yk) &=a_1\left(\cos(2\theta)\; y\tfrac{\partial}{\partial y} 
	+\tfrac{\sin(2\theta)}{2}\;\tfrac{\partial}{\partial \theta}\right)\psi_0^{N}(a_yk)\\
	&=a_1c_3(0)\cdot\cos(2\theta)(2-\delta)y^{2-\delta}. 
	\end{align*}
Hence \eqref{pz} holds for $\ell=1$.

Assuming that \eqref{pz} holds for $\ell$,
we deduce using Lemma \ref{ml} that
	\begin{align*}
	&\psi_{\ell+1}^N(a_yk) = a_{\ell+1}q_\ell \tilde H(y^{2-\delta}M_\ell(\theta)) +b_{\ell+1}q_{\ell-1}y^{2-\delta}M_{\ell-1}(\theta)\\
	&\quad =(-2\ell-1)q_\ell (\cos(2\theta)(2-\delta)M_\ell(\theta) + \tfrac{\sin(2\theta)}{2}
M'_\ell(\theta))y^{2-\delta} \\
	&\quad\quad+ \ell^2(\ell+1-\delta)q_{\ell-1}(\delta+\ell-1)
y^{2-\delta}M_{\ell-1}(\theta).
	\end{align*}
By Lemma \ref{ml},
$$M_{\ell+1}(\theta)=(-2\ell-1)\cos(2\theta)M_\ell(\theta) -\ell^2M_{\ell-1}(\theta)$$
and
\begin{multline*}
	(\ell+2)M_{\ell+1}(\theta)=
2(-2\ell-1)\cos(2\theta)M_{\ell}(\theta)+(-2\ell-1)\tfrac{\sin(2\theta)}{2}M_\ell'(\theta)\\
+\ell^2(\ell-1)M_{\ell-1}(\theta).
\end{multline*}
Therefore we have $$ \psi_{\ell+1}^{N}(a_yk) =q_\ell (\ell+2-\delta) M_{\ell+1}(\theta) y^{2-\delta}=q_{\ell+1} M_{\ell+1}(\theta)y^{2-\delta},$$
proving \eqref{pz} for $\ell+1$.

By the formula of $\|\psi_{\ell}\|_2$ given in Lemma \ref{vbdd}, and using
 $|{\bf P}_\ell(t)|\le 1$ for all $t\in [-1,1]$
and ${\bf P}_\ell(-1)=(-1)^{\ell}$, we finish the proof for $n=3$.
 \end{proof}

\section{Uniform thickening for $\phi_\ell$'s}\label{th}
We continue the notations $G=NAK$, $\G$, $\phi_\ell$, etc. from section \ref{s:h}.
Assume that $\G$ is geometrically finite with $\tfrac{n-1}{2}<\delta< n-1$ and
that $\G\ba \G N$ is closed in $\G\ba G$.
In this section, we approximate the integral $\phi_\ell^N(a_y)$
as the inner product $\la a_y \phi_\ell, \rho_{\eta, \e}\ra=\int_{\G\ba G}\phi_\ell (ga_y) \rho_{\eta, \e}(g)dg$
for a suitable test function $\rho_{\eta,\e}$.
This step enables us to relate the horospherical average $\phi_{\ell}^N(a_y)$ with the spectral
decomposition of $L^2(\G\ba G)$ stated in Theorem \ref{lp}.

Fix a fundamental domain $F_0$ for $N\cap \G$ in $\br^{n-1}$ 
and choose a compact
fundamental domain ${ F}_\Lambda\subset F_0$ for $(N\cap \G)\ba \Lambda(\G)- \{\infty\} $
given by Lemma \ref{bow}.

\begin{lem}\label{es} Fix $\ell \ge 0$. Suppose $\G\ba \G N$ is not compact.
 For any open subset $J\subset F_0 $ containing ${\mathcal F}_\Lambda$,  we have for all $0<y <1$,
\begin{enumerate}
\item $ \int_{J^c} \phi_0(n_x a_y)\; dx \ll y^{\delta}$ with the implied constant independent of $\ell$;
\item $ |\int_{F_0-J} \phi_{ \ell}(n_x a_y)\; dx|\ll (\ell+1)^{(n-2)/2} y^\delta$
with the implied constant independent of $\ell$.
\end{enumerate}
\end{lem}
\begin{proof}
Let $n=2$. Then the non-compactness assumption on $\G\ba \G N$ implies that
$\infty\notin \Lambda(\G)$ and hence
 $\e_0:=\inf_{x\notin J, u\in \Lambda(\G)} |x-u|$ is positive.
Then by the change of variable $w=\frac{x-u}y$, 
 we have

\begin{align*} \int_{F_0-J} \phi_0(n_x a_y)\; dx &\le  2 y^{1-\delta}\int_{u\in \Lambda(\G)}(u^2+1)^\delta d\nu_i(u)
\cdot \int_{w=\e_0/ y}^\infty \left( \tfrac{1}{w^2+1}\right)^\delta dw .\end{align*}

The latter integral can be evaluated explicitly as an incomplete Beta function which has known asymptotics:
$$
\int_{\e_0/y}^{\infty} \left( \tfrac{1}{w^2+1}\right)^{\delta} dw = c\ \beta_{y^2 / \e_0^2}( \delta-1/2, 1-\delta),
$$
where 
$\beta_z(\alpha,\beta)\ll z^{\alpha}.$
Hence
$$ \int_{F_0-J} \phi_0(n_x a_y)\; dx \ll y^{1-\delta}\cdot  y^{2(\delta-1/2)}=y^\delta .$$
This proves (1) for $n=2$, and we refer
to \cite[Proposition 3.7]{KO} for $n=3$.

For the second claim, note that, by Theorem \ref{exp1},
	\begin{align*} \left|\int_{x\in F_0-J} \phi_\ell(n_xa_y)\; dx\right|
	&\leq \int_{x\in F_0-J} \left|\phi_\ell(n_x a_y)\right|\; dx
	 \\ & \ll  (\ell +1)^{(n-2)/2} \int_{F_0-J} \phi_0(n_x a_y)\; dx \end{align*}
since $\phi_0$ is a positive function. 
Hence the claim (2) follows from (1).
% for $k_{\mu_1, \mu_2, \theta}=e$.
%The general case follows by Corollary \ref{ub2} since $|{\bf P}_\ell (-\cos 2\theta)|\le 1$ for all $\theta$ and$\ell\ge 0$.
\end{proof}

For $m\ge 0$, the associated Legendre function  ${\bf P}_\ell^m(x)$ is defined by the following:
 \begin{equation}\label{aleg} {\bf P}_\ell^m(x)=(-1)^m(1-x^2)^{m/2}\tfrac{d^m}{dx^m} ({\bf P}_\ell(x)). \end{equation}
We set ${\bf P}_\ell^{-m}(x):=(-1)^m \tfrac{(\ell-m)!}{(\ell+m)!}{\bf P}_\ell^m(x)$.

 Let $U_\e$ denote the $\e$-neighborhood of $e$ in $G$ for any $\e>0$.
The following lemma controls the Lipschitz constants of $\phi_\ell$'s for the action of $K$:
 \begin{lem} \label{Le} Let $\ell \ge 0$.
 For all sufficiently small $\e>0$, 
$$\phi_\ell(n_xa_y k )= (1+O((\ell+1)^{n-1} \e)) \cdot  \phi_\ell(n_xa_y ) $$
with the implied constant independent of $\ell$, $x$, $y>0$ and $k\in K_\e:=K\cap U_\e$.
\end{lem}
\begin{proof} If $n=2$, we have
$\phi_\ell(n_xa_y k_\theta )=e^{2\ell i\theta}\phi_\ell(n_xa_y )$, and hence the claim follows easily as $e^{2\ell i\theta}=1+O((\ell+1) \theta)$
for all $\theta$ small.

Let $n=3$. We set $\mu=\mu_1+\mu_2$.
The operator $\mathcal {\mathcal C}_K$
acts on $M$-invariant functions as follows:
$$-\mathcal {\mathcal C}_K f(\mu, \theta)=
\tfrac{1}{\sin^2 2\theta}\tfrac{\partial^2}{\partial \mu^2}f+
\tfrac{1}{4\sin 2\theta}\tfrac{\partial}{\partial \theta}\left(\sin 2\theta 
\tfrac{\partial}{\partial \theta}f\right) .$$

Since $\mathcal {\mathcal C}_K(\phi_\ell) =  \ell(\ell+1)\phi_{\ell}$,
it follows from the theory of spherical harmonics (cf. \cite{SUG}) that
\begin{equation}\label{h}
\phi_\ell (n_xa_y k_{\mu_1,\mu_2,\theta}) = 
 \sum_{m=-\ell}^{\ell} f_{ \ell}^m (x, y) \cdot Y_\ell^m(\theta, \mu) \end{equation}
where $f_{\ell, m}\in C^\infty(\c \times
\br_{>0})$ and
$$Y_\ell^m(\theta, \mu)=
\tfrac{\sqrt{(2\ell +1 )(\ell-m)!}}{\sqrt{4\pi(\ell+m)!}} \cdot 
{\bf P}_\ell^m( \cos( 2\theta)) \cdot e^{i m\mu}.$$

We have for $|\mu|<\e$,
 $e^{i m\mu}=1+O(\ell \e)$ as $|m|\le \ell$. 
Also, from the properties of the associated Legendre functions, we deduce that for $|\theta|< \e$ and $|m|\le \ell$,  
$${\bf P}_\ell^m(\cos 2\theta)=(1+O(\ell\e )) {\bf P}_\ell^m(1)  .$$

Therefore for $|\theta|< \e$ and $ |\mu|<\e$,
$$ {\bf P}_\ell^m(\cos 2\theta) \cdot e^{i m\mu}=
(1+O(\ell^2 \e))  {\bf P}_\ell^m(1),$$
and hence
$$Y_\ell^m(\theta, \mu)=
(1+ O(\ell^2 \e)) Y_\ell^m(0,0).$$

It follows that for all $\ell \ge 1$,
\begin{align*} \phi_\ell (n_xa_y k_{\mu_1,\mu_2,\theta})&=(1+ O(\ell^2 \e))
\sum_{m=-\ell}^{\ell} f_{ \ell}^m (x, y) \cdot  Y_\ell^m(0,0) \\ & =(1+ O(\ell^2 \e)) \phi_\ell (n_xa_y ).\end{align*}
\end{proof}

Setting $N^-:=\{n_x^-:=\begin{pmatrix} 1 & 0 \\ x & 1\end{pmatrix} \}$ where $x$ ranges over $\br$ (resp. $\c$) for $n=2$ (resp. $n=3$),
the product map $$N\times A\times M\times N^-\to G$$ is a diffeomorphism at a neighborhood of $e$.
Let $dk$ be the invariant probability measure on $K$. 
For $g=n_xa_yk$, $dg= y^{-n} dxdydk $ defines a Haar measure on $G$.
 Let $\nu$ be a smooth measure on $AMN^-$ such that $dn_x\otimes d\nu = dg$.
Fix a bounded open domain $J\subset F_0$ which contains ${ F}_\Lambda$ and
choose a compactly supported smooth function $0\le \eta\le 1$ on $N$
with $\eta|_{J}=1$. 
If $\infty$ is of rank $n-1$, then set $J=F_0$; hence $\eta=1$ on $F_0$.

Fix $\e_0>0$ so that
the multiplication map
 $${\rm supp}(\eta)\times (U_{\epsilon_0} \cap AMN^-) \to 
{\rm supp}(\eta) \left(U_{\epsilon_0}\cap AMN^- \right) \subset \Gamma\bsl G$$
	is a bijection onto its image. 
	For each $0<\epsilon< \epsilon_0$, let $0\le r_\epsilon\le 1$ be a non-negative smooth function in 
$AMN^-$ whose support is contained in 
		$W_\epsilon:= (U_\epsilon\cap A) M (U_{\epsilon_0} \cap N^-) $ and 
		$\int_{W_\epsilon}r_\epsilon\; d\nu=1\;.$
		
	We define the function $\rho_{\eta, \epsilon}$ on $\Gamma\bsl G$ as follows: for $g=n_xa_ymn_{x'}^-$
\begin{equation*}\rho_{\eta, \epsilon}(g)= \begin{cases} \eta(n_x)\cdot r_\epsilon(a_ym n_{x'}^-)
&\text{ for $g\in {\rm supp}(\eta)W_\e$ }\\ 0 &\text{ for $g\notin
{\rm supp}(\eta)W_\e$}. \end{cases}\end{equation*}

\begin{prop}\label{l:approx_matrix_coeff}
For any $\ell\in \z_{\ge 0}$, we have for all $0<\e \ll 1$ and $y>0$,
	\begin{multline*}\phi_\ell^N(a_y) =\\ 
\left<a_y.\phi_\ell, \rho_{\eta, \epsilon}\right> + 
O( \e (\ell+1)^{\tfrac{n-2}2}  y^{n-1-\delta} + (\ell+1)^{\tfrac{3n-4}2}  y^{n-\delta} ) +O_\eta( (\ell +1)^{\tfrac{n-2}2} y^{\delta})  \end{multline*}
with the implied constants independent of $\ell$.
 \end{prop}
\begin{proof}
Let $h=a_{y_0} n^-_{x} m \in W_\e$.
Then for $n\in N$ and $y>0$,
we have $$nha_y =n a_{y y_0} n^-_{yx} m .$$

As the product map $A\times N\times K\to G$ is a diffeomorphism and
hence a bi-Lipschitz map in a neighborhood of $e$, there exists $q \ge 1$ such that
the $\e$-neighborhood of $e$ in $G$
is contained in the product  $A_{q \e}N_{q\e} K_{q\e}$ for all small $\e>0$.
Therefore we may write
$$n_{yx}^-=  a_{y_1} n_{x_1} k_1\in A_{q y \e_0}N_{q y\e_0} K_{qy \e_0}$$
and hence
\begin{multline*}
 nha_y= na_{yy_0y_1} n_{x_1}k_1 m \\= n(a_{yy_0y_1} n_{x_1}  a_{yy_0y_1}^{-1}) a_{yy_0y_1}    k_1 m=
n n_{x_1 yy_0y_1} a_{yy_0y_1}  k_1 m .
\end{multline*}
By
Lemma \ref{Le},
$$ \phi_\ell (n_x h a_y)=\phi_\ell (n_x n_{x_1 yy_0y_1} a_{yy_0y_1} )  (1+O((\ell+1)^{n-1} y )).$$
By Theorem \ref{ub2}, we may write
 $\phi_\ell^N(a_y)=c_n(\ell) y^{n-1-\delta}$ with $c_n(\ell)\ll (\ell+1)^{(n-2)/2}$.
Hence, using Lemma \ref{es},
\begin{align*} &\int_{(N\cap \G)\ba N} \phi_\ell (n_x h a_y) \cdot \eta(n_x)  dx\\
&
=(1+O((\ell+1)^{n-1} y)) \int_{N\cap \G\ba N} \phi_\ell (n_x n_{x_1 yy_0y_1} a_{yy_0y_1} )  \cdot \eta(n_x) dx.
\end{align*}

Since $x_1 yy_0y_1=1+O(y^2)$, 
\begin{align*} &\int_{(N\cap \G)\ba N} \phi_\ell (n_x h a_y) \cdot \eta(n_x)  dx\\
\\ 
&= (1+O((\ell+1)^{n-1} y)) \int_{N\cap \G\ba N} \phi_\ell (n_x  a_{yy_0y_1} )  \cdot (\eta(n_x)+O(y^2)) dx
 \\
&=  \int_{N\cap \G\ba N} \phi_\ell (n_x  a_{yy_0y_1} )  \cdot \eta(n_x) dx + O( (\ell+1)^{\tfrac{3n-4}2}
 y^{n-\delta} ) \\& + O_\eta(y^2+(\ell+1)^{n-1} y^3) (\ell+1)^{\tfrac{n-2}2} y^{n-1-\delta})
\\
 &= c_n({\ell}) y^{n-1-\delta} (1+O(\e)) + 
O((\ell+1)^{\tfrac{3n-4}2}  y^{n-\delta}) + O_\eta((\ell+1)^{\tfrac{n-2}2} (y^{n+1-\delta} +y^{\delta} )) 
\\
&= c_n({\ell})  y^{n-1-\delta} +O( \e (\ell+1)^{\tfrac{n-2}2}  y^{n-1-\delta} + (\ell+1)^{\tfrac{3n-4}2}  y^{n-\delta} ) +O_\eta((\ell+1)^{\tfrac{n-2}2}y^{\delta}) \end{align*}
 as $|y_0-1|=O(\e)$ and $|y_1-1|=O(y\e)$.

As $\int r_\e d\nu(h)=1$, we deduce from Lemma \ref{es} that
\begin{align*}
& \langle a_y\phi_\ell, \rho_{\eta, \e} \rangle =\int_{W_\e} r_\e(h)
\left(\int_{N\cap \G\ba N}\phi_\ell(n_xha_y) \eta(n_x)  \; dx\right)
\,d\nu(h) \\&
= c_n({\ell}) y^{n-1-\delta} +O( \e (\ell+1)^{\tfrac{n-2}2}  y^{n-1-\delta} + (\ell+1)^{\tfrac{3n-4}2}  y^{n-\delta} ) + O_\eta(  (\ell+1)^{\tfrac{n-2}2} y^{\delta}).
\end{align*}
Since $\phi_\ell^N(a_y)=c_n({\ell}) y^{n-1-\delta}$, this proves the claim.
\end{proof}

The above proposition implies:
\begin{cor}\label{ples} For any $\ell \ge 0$,
$$ |\left<a_y. \phi_\ell, \rho_{\eta, \epsilon}\right>|  \ll  ( \ell+1)^{\tfrac{3n-4}2}  y^{n-1-\delta} 
$$where the implied constant is independent of $\ell\ge 0$, $0<\e<1$ and $0<y<1$.
\end{cor}

\section{Equidistribution of a closed horosphere}
As before, let $n=2$ or $3$, and $G=\PSL_2(\br)$, $\PSL_2(\c)$ accordingly.
Let $\G<G$ be a torsion-free geometrically finite discrete subgroup with
$(n-1)/2<\delta<(n-1)$. We assume that $(N\cap \G)\ba N$ is closed. In this section,
we prove the main effective equidistribution theorem for $(N\cap \G)\ba Na_y$ as $y\to 0$.
%Let $\{Z_i: i \}$ be a basis of the Lie algebra of $G$.

Let $\{X_i \}$ be a basis of the Lie algebra of $G$.
For $\psi \in C^\infty (\G\ba G)\cap L^2(\G\ba G)^M$, we consider the following
$L^2$-Sobolev norm $\mathcal S_m(\psi )$:
$$  \mathcal S_m(\psi )=\sum\| X(\psi) \|_2  $$
where the sum is taken over all monomials $X$ in $X_i$'s of order at most $m$.
By Theorem \ref{lp}, we can fix $ (n-1)/2<s_1<\delta$ so that there is no eigenvalue of  $\Delta$
between $s_1(n-1-s_1)$ and $\delta (n-1-\delta)$ in $L^2(\G\ba \bH^n)$. 
That is, no complementary series representation of $G$ with parameter $s_1<s< \delta$ is contained
in $L^2(\G\ba G)$.

\begin{lem}\label{l:matrix_coeff}
For any $\psi_1,\psi_2\in L^2\left(\Gamma\bsl G\right)^M\cap C^\infty(\Gamma\bsl G)$
 and $0< y< 1$, we have
	$$\left<a_y.\psi_1, \psi_2\right> = \sum_{\ell \in \hat K} \left< \psi_1, \phi_\ell\right>\left<a_y.\phi_\ell, \psi_2\right> + 
O\left(y^{n-1-s_1} \cdot \mathcal S_{n-1}(\psi_1) \cdot \mathcal S_{n-1}(\psi_2) \right).$$
\end{lem}

\begin{proof} 
	We write $L^2(\Gamma\bsl G) = V\oplus V^{\perp}$
	where $V=\mathcal H_\delta$ is the complementary series representation of $G$ of parameter $\delta$.
By Theorem \ref{lp} and the choice of $s_1$, the orthogonal complement $V^\perp$ does not contain any complementary series with parameter
 $s>s_1$. Since
$V^M=\oplus_{\ell \in \hat K} \c \phi_\ell$,
we can write 
	$$\psi_1 = \sum_{\ell \in \hat K} \left<\psi_1, \phi_\ell\right>\phi_\ell + \psi_1^{\perp}$$
with $\psi_1^{\perp}\in V^\perp$. Hence
	\begin{align*}
	\left<a_y.\psi_1, \psi_2\right> & = \sum_{\ell \in \hat K} \left<\psi_1, \phi_\ell\right>\left<a_y.\phi_\ell, \psi_2\right> +\left<a_y.\psi_1^{\perp}, \psi_2\right> .
	\end{align*}

On the other hand, since $V^\perp$ does not contain any complementary series with parameter
 $s>s_1$, we have
$$|\left< a_y. \psi_1^\perp ,\psi_2\right>| \ll  y^{n-1-s_1}\cdot  \mathcal S_{n-1}(\psi_1) \cdot \mathcal S_{n-1}(\psi_2)$$
(see \cite[Prop. 3.3]{KO1} and \cite[Cor. 5.6]{KO}). 
\end{proof}

We refer to \cite[Lem 6.5]{KO} for the next lemma: 
\begin{lem}\label{l:hat_psi}
For $\psi\in C_c^\infty(\Gamma\bsl G)$, there exists $\widehat \psi\in C_c^\infty(\Gamma\bsl G)$ such that
	\begin{enumerate}
	\item for all small $\epsilon>0$, $h\in U_\epsilon$, and $g\in \Gamma\bsl G$,
		$$\left|\psi(g)-\psi(gh)\right|\leq \epsilon\cdot\widehat \psi(g).$$
		
	\item for any $i\ge 1$, $\mathcal S_i(\widehat \psi)\ll \mathcal S_{2n-1}(\psi)$  with the implied constant depending only on ${\rm supp}(\psi)$.
	\end{enumerate}
\end{lem}

\medskip 

\begin{lem}\label{haha}
 For any $\psi\in C_c^\infty(\Gamma\bsl G)$,  $\ell \ge 0$, and $i\ge 1$, we have
$$ \left| \left<\psi, \phi_\ell\right> \right|\ll (\ell+1) ^{-2i} \|\mathcal {\mathcal C}_K^i \psi\|_2.$$
In particular, for each $i\ge 1$,
$$\sum_{\ell \in \hat K} ( |\ell| +1)^{2i}\cdot \left| \left<\psi, \phi_\ell\right> \right|\ll   \|\mathcal {\mathcal C}_K^{i+1} \psi\|_2. $$ 
\end{lem}
\begin{proof}
Recall that
the operator $\mathcal {\mathcal C}_K$ acts on each $V_\ell$ as a scalar $\alpha_\ell$, where
$\alpha_\ell =-4\ell^2$  for $n=2$ and and $\alpha_\ell=\ell(\ell+1)$ for $n=3$.
Moreover, for any smooth vectors $v,w$,
 $|\la \mathcal C_K v, w\ra|=|\la v, \mathcal C_K w\ra|$, as $\mathcal C_K$ is skew-adjoint for $n=2$ and
$\mC_K$ is adjoint for $n=3$.
Therefore
  $$|\left<\mathcal {\mathcal C}_K^i \psi, \phi_\ell\right>|=
|\left<\psi, \mathcal {\mathcal C}_K^i  \phi_\ell\right>|= |\alpha_\ell^i| \cdot | \left<\psi, \phi_\ell\right>|.$$
Hence  for all $i \ge 1$,
$$
|\left<\psi, \phi_\ell\right>|\le |\alpha_\ell |^{-i} \cdot  \|\mathcal {\mathcal C}_K^i \psi\|_2 .$$
\end{proof}

By Theorem \ref{ub2}, we may write
$$\phi_{\ell}^N(a_y) 
	= c_n(\ell) \cdot y^{n-1-\delta}$$
where \be \label{cnl} c_n(\ell)=(-1)^{(n-2)\ell}  \cdot c_n(0)
\cdot \tfrac{\sqrt{\Gamma(n-1 -\delta)\Gamma(\ell+\delta)}}{\sqrt{\Gamma (\delta) 
\Gamma(\ell+n-1-\delta)}} \cdot   \sqrt{2(n-2)\ell+1} .\ee

\begin{thm}\label{m1} 
Fix $0<\bold s_\G<\delta-s_1$. For any $\psi\in C_c^{\infty}(\Gamma\bsl G)^M$,
$$\psi^N(a_y) =\sum_{\ell \in \hat K}  c_n(\ell) \left<\psi, \phi_\ell\right> y^{n-1-\delta}
+\mathcal S_{2n-1}(\psi) O( y^{n-1-\delta+ \tfrac{2{\s}_\G}{(2n+1)}} ).$$
\end{thm}
%: analogous to theorem 6.1
\begin{proof}  

Fix $\psi\in C_c^{\infty}(\Gamma\bsl G)^M$.
When $\infty$ has rank $n-1$, we set $J=F_0$.
In other cases, it was shown in \cite{KO} that
there exists a bounded open subset $J$ of $F_0$ such that
$\psi(n_xa_y)=0$ for all $x\in F_0-J$ and all $0<y<1$. 
We assume that $J$ contains ${ F}_\Lambda$, which is a bounded fundamental domain for the action of $(N\cap\G)$ in $\Lambda(\G)-\{\infty\}$. 

Choose a non-negative function $\eta\in C_c^\infty (N\cap \G \ba N)$ such that $\eta|_J=1$. Then
	$$I_\eta(\psi)(a_y): = \int_{(N\cap \Gamma)\bsl N} \psi(n_xa_y) \eta(n_x)\; dx =
  \psi^N(a_y)\;.$$

Let $\e_0, W_\e, r_\e, \rho_{\eta, \e}$ be as defined in section \ref{th}
with respect to  $J$ and $\eta$. 
Since $r_\e$ is the approximation of the identity in $A$ direction, $\mathcal S_{n-1}(\rho_{\eta, \e})=O_\eta(\e^{(-2n+1)/2})$.
 For any $0< y< 1$, and any small $\epsilon>0$, we have (see the proof of Prop. 6.6 in \cite{KO})
	\begin{equation}\label{l:aprox_matrix_coeff} \left|I_\eta(\psi)(a_y)-\left<a_y.\psi, \rho_{\eta, \epsilon}\right>\right| \ll (\epsilon+y)\cdot I_\eta(\widehat \psi)(a_y).\end{equation}

Setting $\psi_0(g) := \psi(g)$, we define for $1\leq i\leq k$, inductively
	$$\psi_i(g) :=\widehat\psi_{i-1}(g)$$
where $\widehat\psi_{i-1}$ is given by Lemma \ref{l:hat_psi}. 

Let $k$ be an integer bigger than $1+\tfrac{( n-1-\delta)(2n+1)}{2(\delta-s_1)} $. 

Applying Lemma \ref{l:hat_psi} to each $\psi_i$, we obtain for $0\leq i\leq k-1$, 
	\begin{align*}
	I_{\eta}(\psi_i)(a_y) &= \left<a_y.\psi_i, \rho_{\eta, \epsilon}\right> + 
O\left((\epsilon+y)\cdot I_\eta(\widehat\psi_i)(a_y)\right)
	\\&=\left<a_y.\psi_i, \rho_{\eta, \epsilon}\right> + O\left((\epsilon+y)\cdot I_\eta(\psi_{i+1})(a_y)\right)
	\end{align*} and
	$$I_\eta(\psi_k)(a_y) = \left<a_y. \psi_k, \rho_{\eta, \epsilon}\right> +
O_\eta\left((\epsilon+y)\mathcal S_{n-1}(\psi_k)\right).$$
 
By Lemma \ref{haha},
 $$|\la \psi_i, \phi_\ell\ra |=O(|\ell|+1)^{-(2n-2)} \mathcal S_{2n-2}(\psi_i)$$ 
and by Lemma \ref{l:hat_psi}, $\mathcal S_{j}(\psi_i)\ll \mathcal S_{2n-1}(\psi)$ for all $j\ge 1$.

Since $|\left<a_y. \phi_\ell, \rho_{\eta, \epsilon}\right>|  \ll   (|\ell| +1)^{(3n-4)/2}  y^{n-1-\delta} $
by Corollary \ref{ples}, 
we have
$$\sum_{\ell \in \hat K} \left<\psi_i, \phi_\ell\right>\left<a_y.\phi_\ell, \rho_{\eta, \epsilon}\right>=
 O (\mathcal S_{2n-1}(\psi_i) \cdot y^{n-1-\delta} ).$$

By Lemma \ref{l:matrix_coeff},
  we deduce that for each $1\leq i \leq k-1$, 
	\begin{align*}
	\left<a_y.\psi_i, \rho_{\eta, \epsilon}\right> &=
\sum_{\ell \in \hat K} \left<\psi_i, \phi_\ell\right>\left<a_y.\phi_\ell, \rho_{\eta, \epsilon}\right> +
 O\left(y^{n-1-s_1}\cdot\mathcal S_{n-1}(\psi_i)\cdot\mathcal S_{n-1}(\rho_{\eta, \epsilon})\right)
	\\&= O (\mathcal S_{2n-1}(\psi) \cdot y^{n-1-\delta} ) +O ( y^{n-1-s_1}\cdot \mathcal S_{n-1}(\psi_i) \mathcal S_{n-1}(\rho_{\eta, \epsilon}) )
	\\&= \mathcal S_{2n-1}(\psi) \cdot O( y^{n-1-\delta} +\epsilon^{-(2n-1)/2}y^{n-1-s_1}).
	\end{align*}
Hence for any $0<y<\epsilon$, using Proposition \ref{l:approx_matrix_coeff}, we deduce
	\begin{align*}&
	I_\eta(\psi)(a_y) = \left<a_y.\psi, \rho_{\eta, \epsilon}\right> +
\sum_{j=1}^{k-1}O\left(\left<a_y.\psi_j, \rho_{\eta, \epsilon}\right>(\epsilon+y)^j\right) +
 O_\psi ((\epsilon+y)^k)
\\&=\left<a_y.\psi, \rho_{\eta, \epsilon}\right> 
 + \mathcal S_{2n-1}(\psi) O(\epsilon\cdot y^{n-1-\delta} + \epsilon^{-(2n-1)/2} y^{n-1-s_1} +\epsilon^k)
	\\&=\sum_{\ell \in \hat K} \left<\psi, \phi_\ell\right>\left<a_y.\phi_\ell,
 \rho_{\eta, \epsilon}\right> +
\mathcal S_{2n-1}(\psi) O(\epsilon\cdot y^{n-1-\delta} + \epsilon^{-(2n-1)/2} y^{n-1-s_1} +\epsilon^k)
 \\&= \sum_{\ell \in \hat K} \left<\psi, \phi_\ell\right>  c_n(\ell) y^{n-1-\delta} +
 \\ & \sum_{\ell \in \hat K} \left<\psi, \phi_\ell\right>  \left( O( \e (|\ell|+1)^{\tfrac{n-2}2}  y^{n-1-\delta} + (|\ell|+1)^{\tfrac{3n-4}2}  y^{n-\delta} ) 
 +
O_\eta( (|\ell | +1)^{(n-2)/2}y^{\delta}) \right) 
\\&\quad\quad +
\mathcal S_{2n-1}(\psi) O(y^{\delta} +\e y^{n-1-\delta}+ \epsilon^{-(2n-1)/2} y^{n-1-s_1} +\epsilon^k).
\end{align*}

Since
$| \left<\psi, \phi_\ell\right> |\ll ( |\ell | +1)^{-(2n-2)}O(\mathcal S_{2n-2}(\psi))$,
we have
$$\sum_{\ell \in \hat K} \left<\psi, \phi_\ell\right> (|\ell |+1)^{(3n-4)/2}=O(1),\quad
\sum_{\ell \in \hat K} \left<\psi, \phi_\ell\right> (|\ell | +1)^{(n-2)/2}=O(1).$$

Hence
we deduce
\begin{multline*}I_\eta(\psi)(a_y)=
\sum_{\ell \in \hat K} \left<\psi, \phi_\ell\right> c_n(\ell) y^{n-1-\delta} \\  +
\mathcal S_{2n-1}(\psi) O(\e y^{n-1-\delta}+y^{n-\delta} +y^\delta +\epsilon^{-(2n-1)/2} y^{n-1-s_1} +\epsilon^k).\end{multline*}

By equating $\epsilon\cdot y^{n-1-\delta} $ and $ \epsilon^{-(2n-1)/2} y^{n-1-s_1}$,
we put $\e=y^{2(\delta-s_1)/(2n+1)}$ and obtain
$$I_\eta(\psi)(a_y) =\sum_{\ell \in \hat K} \left<\psi, \phi_\ell\right> c_n(\ell) y^{n-1-\delta}
+\mathcal S_{2n-1}(\psi) O(y^{n-1-\delta +\frac{2(\delta-s_1)}{2n+1}}).$$
\end{proof}

\section{Comparing main terms from different approaches}\label{bms}
The coefficient $\sum_{\ell\in \hat K}c_n(\ell) \la   \psi ,\phi_\ell\ra$ of the main term in Theorem \ref{m1} is related to
the space average of $\psi$ with respect to the Burger-Roblin measure (which we will call
the BR measure for short) by \cite{Ro} and \cite{KO} (also see \cite{OS}). 

Recall that $\phi_0=\phi_0^\G$ is given by
$$\phi_0(x+jy) = \int_{\br^{n-1}} \left(\tfrac{(|u|^2+1)y }{|x-u|^2 +y^2}\right)^\delta d\nu_j(u)$$
for the Patterson-Sullivan measure $\nu_j=\nu_j^\G$ on the boundary $\partial(\bH^n)$.
 Note that $\phi_0^\G(j)=|\nu_j^\G|.$
As before, we normalize $\nu_j$ so that $\|\phi_0\|_2=1$. 

For $\xi\in \partial(\bH^n)$ and $z_1, z_2\in \bH^n$,
recall the Busemann function:
$$\beta_\xi(z_1,z_2)=\lim_{s\to \infty}d(z_1,\xi_s)-d(z_2,\xi_s)$$
where $\xi_s$ is a geodesic ray tending to $\xi$ as $s\to \infty$.
 Using the identification of $\T^1(\bH^n)$ and $G/M$,
we give the definition of 
the Bowen-Margulis-Sullivan measure $m^{\BMS}$  on $\Gamma\ba G/M$.
For $u\in \T^1(\bH^n)$, we denote by $u^+$ and $u^-$ the forward and the backward endpoints
of the geodesic determined by $u$, respectively.
The correspondence $u\mapsto (u^+, u^-, \beta_{u^-}(j, \pi(u)))$ gives a homeomorphism between
the space $\T^1(\bH^n)$ with  
 $(\partial(\bH^n)\times \partial(\bH^n) - \{(\xi,\xi):\xi\in \partial(\bH^n)\})  \times \br $
where $\pi: G\to G/K=\bH^n$ is the canonical projection.
Define the measure $\tilde m^{\BMS}$ on $G/M$:
$$d \tilde m^{\BMS}(u)=e^{\delta \beta_{u^+}(j, \pi(u))}\;
 e^{\delta \beta_{u^-}(j,\pi(u)) }\;
d\nu_j(u^+) d\nu_j(u^-) dt $$
where $t=\beta_{u^-}(j, \pi(u))$.
This measure is left $\G$-invariant and hence induces a measure $m^{\BMS}$ on $\G\ba G/M$.

Roblin obtained the following interesting identity in his thesis \cite{RoT}:
\begin{thm}[Roblin] \label{rot} For $\delta >(n-1)/2$,
 $$\|\phi_0\|_2^2=|m^{\BMS}|\cdot\int_{\br^{n-1}} \tfrac{dx}{(1+|x|^2)^\delta} .$$
\end{thm}

As we have normalized $\nu_j$ so that $\|\phi_0\|_2=1$ and $\phi_0(j)=|\nu_j|$, we deduce
$$\frac{1}{|m^{\BMS}|}=\int_{\br^{n-1}} \tfrac{dx}{(1+|x|^2)^\delta}.$$

To describe the equidistribution result of $(N\cap \G)\ba N a_y$ from \cite{Ro}, we recall
 the measure $\tilde m^{\BR}_{N}$ defined in the introduction:
for $\psi\in C_c( G/M)$,
$$\tilde m^{\BR}_N(\psi)=
\int_{KAN}\psi(k a_yn_x)y^{\delta -1} dx dyd\nu_j(k(0)) .$$
%with $o$ being the origin in the boundary $\partial(\bH^2)=\br\cup\{\infty\}$.
The BR measure $m^{\BR}_{N}$ (associated to the stable horospherical subgroup $N$) 
is the measure on $\G\ba G/M$ induced from $\tilde m^{\BR}_N$.

We define the measure $\mu^{\PS}_N$ on $N$ by
 $$d\mu^{\PS}_N(n_x)=e^{-\delta\beta_{x} (j, x+j)} d\nu_j(x)={(1+|x|^2)^{\delta}} d\nu_j(x).$$ 
%where $\beta_\xi(x_1,x_2)$ denotes 
%the Busemann function, i.e., the signed distance between horocycles based at $\xi$ passing through $y_1$ and $y_2$.
This induces a measure on $(N\cap \G)\ba N$ for which we use the same notation $\mu_N^{\PS}$. 
Since $\mu^{\PS}_N$ is supported in $(N\cap \G)\ba (\Lambda(\G)-\{\infty\})$, which is compact, we have
 $ \mu^{\PS}_N ((N\cap \Gamma) \ba N)<\infty$.

%We recall the Burger-Roblin measure $m_N^{\BR}$ defined in the introduction.
The following is proved by Roblin \cite{Ro} when$(N\cap \G)\ba N$ is compact and 
in \cite{OS} in general.
\begin{thm} \label{ros} Let $\delta>0$ and $(N\cap \G)\ba N$ closed.
 For any $\psi\in C_c(\Gamma\ba G)^M$, 
$$ \lim_{y\to 0} y^{\delta-n+1} \cdot \psi^N(a_y) = \tfrac{\mu^{\PS}_N (N\cap \Gamma \ba N)}{|m^{\BMS}|} m^{\BR}_N(\psi) .$$
\end{thm}

Comparing the coefficients of the main terms of Theorem \ref{ros} and Theorem \ref{m1}, and using Theorem \ref{rot},
we deduce
the following identity of the Burger-Roblin measure considered
as a distribution on $\G\ba G$: 
\begin{thm}\label{brd} Let $\delta>(n-1)/2$. For any $\psi\in C_c^\infty(\Gamma\ba G)$,
 $$\kappa_{\G} \cdot m^{\BR}_N(\psi) =\sum_{\ell \in \hat K}  c_n(\ell)\left<\psi, \phi_{\ell}\right>$$
with $c_n(\ell)$ as in \eqref{cnl} and $\kappa_{\G}= \int_{\br^{n-1}} \tfrac{dx}{(1+|x|^2)^\delta}
\cdot  \int_{n_x\in (N\cap \G)\ba N} {(1+|x|^2)^{\delta}}{{d\nu_j(x)}}.$
\end{thm}

Now Theorem \ref{main2} is a direct consequence of Theorem \ref{brd} and Theorem \ref{main}.

\section{Application to counting in sectors}

Let $n=2$ or $3$. Let $Q$ be a real quadratic form of signature $(n,1)$ and $v_0\in \br^{n+1}$ be a non-zero vector
such that $Q(v_0)=0$. 
 Let $\G_0$ be a geometrically finite subgroup of the identity component
of $ \SO_Q(\br)$. 
Suppose that $\delta>(n-1)/2$ and that $v_0\G_0$ is discrete.

 Let $\|\cdot \|$ be {\it any} norm in $\br^{n+1}$ and set $B_T:=\{v\in \br^{n+1}: \|v\|<T\}$.
 Let $G=\PSL_2(\br)$ if $n=2$ and $\PSL_2(\c)$ if $n=3$. Let $\iota: G\to \SO_Q(\br)$ be a representation so that
the stabilizer of $v_0$ in $G$ via $\iota$ is
$NM$. Let $\Gamma:=\iota^{-1}(\Gamma_0)$.

\subsection{Counting I}
For $g\in G$, we write $\kappa (g)$ for the $K$-coordinate of $g$ in the Iwasawa decomposition $G=NAK$.
As before, $M$ denotes the centralizer of $A$ in $K$.
Fixing a function $f$ on $M\ba K$,
define
the counting function $F_T$ on $\G\ba G$ by
$$F_T(g)=\sum_{\gamma\in N\cap \Gamma\ba \Gamma} \chi_{B_T}(v_0 \gamma g) f([\kappa (\gamma g)])$$
where $\chi_{B_T}$ denotes the characteristic function of $B_T$. Since $\kappa (g)=\kappa(ng)$ for any $n\in N$ and $g\in G$,
$F_T$ is well-defined.

For $k\in K$ and $\psi\in C_c(G)$,
define $\psi^k\in C_c(G)^M$ by
$$\psi^k(g)=\int_{m\in M} \psi(g m k) dm $$
where $dm$ denotes the probability Haar measure of $M$.
Similarly, for $\Psi\in C_c(\G\ba G)$, we set $\Psi^k(g)=\int_{m\in M} \Psi(g m k) dm .$
\begin{lem}\label{uf}
 For $\Psi\in C_c(\G\ba G)$ and for any Borel function $f$ on $K$, we have
$$\la F_T, \Psi\ra =\int_{k\in  M\ba K} f([k]) \int_{y>\|v_0k\| T^{-1}}
 \left(\int_{(N\cap \G)\ba N} \Psi^k(n_xa_y) \; dx\right) y^{-n} dydk$$
where $dk$ denotes the probability Haar measure on $K$, also understood as the invariant measure on $M\ba K$.
\end{lem}
\begin{proof}
We compute
\begin{align*}&\la F_T, \Psi\ra=
\int_{\G\ba G} \sum_{\gamma\in N\cap \Gamma\ba \Gamma} \chi_{B_T}(v_0 \gamma g) f(\kappa (\gamma g)) \Psi(g) \; dg\\ &=
 \int_{(N\cap \Gamma)\ba G} \chi_{B_T}(v_0 g) f(\kappa ( g)) \Psi(g) \; dg\\ &=
\int_{a_yk\in AK} \chi_{B_T}(v_0 a_yk)  f(k) \left(\int_{(N\cap \G)\ba N} \Psi(n_xa_yk) \; dx\right) y^{-n} dydk\\ &=
\int_{k\in M\ba K} \int_{y>\|v_0k\| T^{-1}}
 f(k) \left(\int_{(N\cap \G)\ba N} (\int_{m\in M}\Psi(n_xa_y m k)dm) \; dx\right) y^{-n} dydk .
\end{align*}
 \end{proof}

By Theorem \ref{main2}, for any $\Psi\in C^\infty(\G\ba G)^M$ and $k\in M\ba K$, we have, as $y\to 0$,
\begin{align*}\int_{(N\cap \Gamma)\ba N} \Psi^k (n_x a_y ) dx &= 
 \kappa_\G \cdot
{m_N^{\BR}(\Psi^k)} \cdot y^{n-1-\delta}
+\mathcal S_{2n-1}(\Psi) O(y^{(n-1-\delta)+\tfrac{2{\bf s_\G}}{2n+1}}) .\end{align*}

Therefore, we deduce from Lemma \ref{uf}:

\begin{thm}\label{fmc} For any $\Psi\in C^\infty(\G\ba G)$ and
 a bounded Borel function $f$ on $M\ba K$, we have, as $T\to \infty$,
 $$\la F_T, \Psi\ra =\tfrac{ \kappa_\G} {\delta} \cdot \left(\int_{k\in M\ba K} \tfrac{m_N^{\BR}(\Psi^k)\cdot f(k) }{\|v_0k\|^\delta} \; dk \right) \cdot  T^\delta +
O(\mathcal S_{2n-1}(\Psi) T^{\delta-\tfrac{2{\bf s_\G}}{2n+1}}). $$
\end{thm}

\subsection{Counting II}
For a left $M$-invariant Borel subset $\Omega\subset K$ and $T>0$,
define
$$S_T(\Omega):=\{ v\in v_0A\Omega :\|v\|<T\}. $$

For a subset $I$ of $\z_{\ge 0}$,
 let $\{\G_d <\G_0: d\in I\}$ be a family of  subgroups of finite index which
 satisfies $\text{Stab}_{\G_0} {v_0} =\text{Stab}_{\G_d} {v_0} $ and
which has a uniform spectral gap, say, ${\s}_0$.

%Theorem \ref{main2} can be used to prove the following counting theorem (this can be done
%precisely the same way as in \cite{KO1}):
Set \begin{equation}\label{exi} \Xi_{v_0}(\G_0, \Omega):=\tfrac{\kappa_{\iota^{-1}(\G_0)}}{\delta} 
\int_{k^{-1}\in \Omega} \tfrac{ d\nu_j^{\Gamma}(k(0))} {\|v_0k^{-1} \|^{\delta}} . \end{equation}

We deduce the following from Theorem \ref{fmc}:

\begin{thm}\label{ec}
Let $\Omega$ be an admissible $M$-invariant Borel subset of $K$ and let $q_\Omega>0$ be as in \eqref{qo}.
 For any $\gamma'\in \G_0$, $$\# (v_0\Gamma_d \gamma'\cap S_T(\Omega))=
\tfrac{\Xi_{v_0}(\G_0, \Omega)}{ [\G_0:\G_d]} 
 T^\delta +O(T^{\delta -\tfrac{8{\s}_0}{n(n+9)(2n+1)q_\Omega}})$$
with the implied constant independent of $d$ and $\gamma'$.
\end{thm}
\begin{proof}
Let $\Gamma:=\iota^{-1}(\Gamma_0)$ and $\gamma_0:=\iota^{-1}(\gamma')$. Recall
that $U_\e$ denotes an $\e$-neighborhood of $e$ in $G$.
By abuse of notation, we use the notation $\Gamma_d$ to denote
$\iota^{-1}(\Gamma_d)$.
%Since $\delta>(n-1)/2$, $\G$ is Zariski dense and hence any proper algebraic subset of $\partial(\bH^n)$ has $\nu_j$-measure zero \cite{FS}.
%Hence $\nu_j(\partial(\Omega^{-1}(0)))=0$.
 Moreover, for all sufficiently small $\e>0$,
 and for $K_\e:=U_\e\cap K$, we have, for $\Omega_{\e+}=\Omega K_\e$ and 
  $\Omega_{\e-}=\cap_{k\in K_\e} \Omega k$,
 \begin{equation}\label{eq:600}
 \nu_j(\Omega_{\e +}^{-1}(0) -  \Omega_{\e-}^{-1}(0))=O( \e^{q'_\Omega}).
 \end{equation}

By the strong wave front lemma \cite[Theorem 4.1]{GOS},  there exists $0<\ell_{0}<1$ such that for $T\gg 1$,
$$S_T(\Omega) U_{\ell_0 \e} \subset S_{(1+\e)T}(\Omega_{\e+}) \quad\text{and}\quad
S_{(1-\e)T}(\Omega_{\e-})\subset\cap_{u\in U_{\ell_0 \e}} S_T(\Omega) u .$$

Let
 $\psi_\e \in C_c^\infty(G)$ be a non-negative function supported in $U_{\ell_0 \e}$ with integral one, and
define the following function of $\G_d\ba G$: $$\Psi_{\G_d,\e}(g)=\sum_{\gamma\in \G_d} \psi_\e(\gamma g) .$$

Define the counting function $F_T^{\Omega}$ on $\G_d\ba G$ by
$$F_T^{\Omega}(g)=\sum_{\gamma\in N\cap \Gamma\ba \Gamma_d} \chi_{S_T(\Omega)}(v_0 \gamma g)=
\sum_{\gamma\in N\cap \Gamma\ba \Gamma_d} \chi_{B_T}(v_0 \gamma g)\chi_\Omega ([\kappa(\gamma g)]) ;$$
this is well-defined as $\G_d\cap N=\G\cap N$ by the assumption on $\G_d$.
 Note that $F_T^{\Omega}(\gamma_0)=\# v_0\Gamma_d \gamma_0\cap S_T(\Omega)$
and that
$$ F_{(1-\e)T}^{\Omega_{\e-}}(\gamma_0 g)\le
 F_T^{\Omega}(\gamma_0) \le F_{(1+\e)T}^{\Omega_{\e+}}(\gamma_0 g) $$
for all $g\in U_{\ell_0\e}$.
On the other hand,
$$
\int_{\G_d\ba G} F_{(1\pm\e)T}^{\Omega_{\e\pm}}(\gamma_0 g) \Psi_{\G_d,\e} (g) dg
=\int_{\G_d\ba G} F_{(1\pm \e)T}^{\Omega_{\e\pm}}( g) \Psi_{\G_d, \e} (\gamma_0^{-1} g) dg.$$

Therefore, if we set $\Psi_{\G_d,\e}^{\gamma_0}(g):= \Psi_{\G_d,\e} (\gamma_0^{-1} g) $, we have
\[  
\la F_{(1-\e)T}^{\Omega_{\e-}} ,\Psi_{\G_d,\e}^{\gamma_0}\ra \le F_{T}^{\Omega}(\gamma_0)
\le\la  F_{(1+\e)T}^{\Omega_{\e+}},\Psi_{\G_d,\e}^{\gamma_0}\ra
\]
where the inner product has taken place in $L^2(\G_d\ba G)$.
Since $\phi_0^{\Gamma}(e)=|\nu_j^\Gamma|$ and $\|\phi_0^\G\|_2=1$,
we note that for all positive integer $d$,
$$\nu_j^{\Gamma_d}=\tfrac{1}{\sqrt{[\G:\G_d]}}\nu_j^\Gamma .$$
Therefore
it follows that $\kappa_{\iota^{-1}(\G_d)}= \frac{1}{\sqrt{[\G:\G_d]}}\kappa_{\iota^{-1}(\G)}$, and 
that for any $\G$-invariant continuous function $f$ on $G/M$,
 $$m^{\BR}_{\G_d, N}(f)=\tfrac{1}{\sqrt{[\G:\G_d]}} m^{\BR}_{\G,N}(f).$$ 
We use the following (see \cite[Prop. 6.2]{OS}, or \cite[Sec. 7]{KO}):
\be\label{osss} \int_{k \in \Omega}
 \tfrac{  m^{\BR}_{\G,N}(\Psi_{\G,\e}^{k})}{\|v_0k\|^\delta}
 dk =\int_{k\in \Omega^{-1}} \tfrac{ d \nu_j^{\G}(k(0))}{\|v_0k^{-1}\|^\delta} \cdot (1+O(\e)). \ee

Since $\text{dim}(G)=n(n+1)/2$, we compute
$\mathcal S_{2n-1}(\Psi_\e)=O(\e^{-(n^2+9n-4)/4})$.
Hence putting these together and using Theorem \ref{fmc}, we have 
\begin{align*} 
&\la F_{(1\pm\e)T}^{\Omega_{\e\pm}}, \Psi_{\G_d,\e}^{\gamma_0} \ra 
 \\&=\tfrac{\kappa_\G \cdot {(1\pm \e)}^\delta }{\delta [\G:\G_d]  }  T^\delta 
\int_{k\in \Omega^{-1}_{\e \pm}}\tfrac{d \nu_j(k(0))}{\|v_0k^{-1}\|^\delta} 
+O(\e T^\delta + \e^{{-(n^2+9n-4)/4}} T^{\delta -\tfrac{2{\s}_0}{2n+1}}) 
 \\ &= 
\tfrac{\kappa_\G \cdot T^{\delta}}{\delta  [\G:\G_d]} \int_{k\in \Omega^{-1}_{\e \pm}} 
\tfrac{ d \nu_j(k(0))}{\|v_0k^{-1}\|^\delta} 
+O(\e T^\delta + \e^{{-(n^2+9n-4)/4}} T^{\delta-\tfrac{2{\s}_0}{2n+1}}) 
\\ &=
 \tfrac{\kappa_\G \cdot T^{\delta}}{\delta  [\G:\G_d]} \int_{k\in \Omega^{-1}} 
\tfrac{ d \nu_j(k(0))}{\|v_0k^{-1}\|^\delta} 
+O(\e^{q_\Omega} T^\delta + \e^{{-(n^2+9n-4)/4}} T^{\delta -\tfrac{2{\s}_0}{2n+1}}) . 
\end{align*}

Hence, by equating $\e^{- (n^2+9n-4)/4}T^{-\tfrac{2{\s}_0}{(2n+1)}}=\e^{q_\Omega}$, we deduce
$$F_T^{\Omega}(e)=\tfrac{\kappa_\G \cdot T^{\delta}}{\delta [\G:\G_d]} \int_{k\in \Omega^{-1}}\tfrac{ d \nu_j(k(0))}{\|v_0k^{-1}\|^\delta} 
+O( T^{\delta -\tfrac{8{\s}_0}{n(n+9)(2n+1)q_\Omega}} ).$$
 This finishes the proof of Theorem \ref{ec}.

\end{proof}

\subsection{}\label{aps}
Let $\P$ be an Apollonian packing as in Theorem \ref{m}. Let 
$$Q(x_1,x_2,x_3,x_4)=2(x_1^2+x_2^2+x_3^2+x_4^2)-(x_1+x_2+x_3+x_4)^2$$
be the Descartes quadratic form, which has signature $(3,1)$.
Let $\mathcal A$ denote the Apollonian group, i.e., the subgroup of $O_Q(\Z)$ generated by 
$$S_1=\begin{pmatrix}  -1&0&0&0 \\2&1&0&0\\
2&0&1&0\\ 2&0&0&1 \end{pmatrix},\quad
 S_2=\begin{pmatrix} 1&2&0&0\\
 0&-1& 0&0 \\
0&2&1&0\\ 0&2&0&1 \end{pmatrix}, $$
$$
 S_3=\begin{pmatrix} 1&0&2&0\\
 0&1&2&0\\ 0&0& -1 &0 \\
 0&0&2&1 \end{pmatrix}, \quad  S_4=\begin{pmatrix} 1&0&0&2\\
 0&1&0&2 \\ 0&0&1&2\\ 0&0&0&-1 \end{pmatrix}. $$

The critical exponent of $\mathcal A$ is equal to $\alpha$, which is
the Hausdorff dimension of the residual set of $\P$ and is a geometrically finite group (cf. \cite{KO}).

In \cite[Sec. 2]{KO}, 
it was shown that there exists a vector $v_0$ with $Q(v_0)=0$, whose coordinates are given by
the curvatures of four mutually tangent circles of $\P$ and that
$$N_T(\P)=\{v\in v_0\mathcal A: \|v\|_{\max}<T \}+3$$
for all $T\gg 1$.

Therefore Theorem \ref{ec} implies, as $q_K=1$: 

\begin{cor} \label{apm} For some $c_\P>0$,
$$N_T(\P)= c_\P \cdot T^{\alpha} +O(T^{\alpha -2{\s}_{\mathcal A}/63}) .$$ 
 \end{cor}
Moreover, if we set $\A_0<\SO(Q)^\circ$ to be a torsion free finite index subgroup of $\A$
and write $v_0\mathcal A$ as the disjoint union $\cup_{i=1}^m v_i\mathcal A_0$,
then
\be\label{cpp} c_\P=\sum_{i=1}^m \Xi_{v_i}(\mathcal A_0, K) \ee
where $\Xi_{v_i}$ is defined as in \eqref{exi}.

On the other hand, it can be deduced from
the main results in \cite{OhShahcircle} that 
$$\lim_{T\to \infty}\frac{N_T(\P)}{T^\alpha}=c_A\cdot \mathcal H_\alpha(\Res(\P))$$
where $c_A>0$ is a constant independent of $\P$ (cf. \cite{Oh} for details).
Therefore Theorem \ref{m} follows from Corollary \ref{apm}.

\end{document}